\documentclass[12pt,longbibliography]{article}

\usepackage[utf8]{inputenc}

\usepackage[T1]{fontenc}

\usepackage[a4paper, margin=2.7cm]{geometry}

\usepackage{amsmath, amsthm, amsfonts, amssymb}
\usepackage{mathrsfs}           % \mathscr font.

\usepackage[skip=10pt plus1pt, indent=10pt]{parskip}
\usepackage[colorlinks=true,linkcolor=blue,citecolor=blue,urlcolor=blue,breaklinks]{hyperref}

\usepackage{cleveref}

\usepackage{bbm}

\usepackage[square,numbers]{natbib}
\usepackage{url}
\newcommand{\abs}[1]{\left\vert#1\right\vert}
\def\R{\mathbb{R}}
\def\S{\mathbb{S}}

\def\d{\,\mathrm{d}}

\def\:{\colon}

\def\R{\mathbb{R}}

\def\1{\mathbbm{1}}

\DeclareMathOperator{\dist}{dist}

\def\d{\,\mathrm{d}}

\def\H{\mathbb{H}}

\def\:{\colon}

\def\lambdaL{\lambda_{\mathrm{L}}}

\newtheorem{thm}{Theorem}[section]
\newtheorem{cor}[thm]{Corollary}
\newtheorem{lem}[thm]{Lemma}
\newtheorem{prp}[thm]{Proposition}

\theoremstyle{definition}

\theoremstyle{remark}
\newtheorem{rem}[thm]{Remark}
\theoremstyle{example}

\def\thetitle{Sharp asymptotic behaviour of symmetric and non-symmetric solutions of the Heat Equation in the Hyperbolic Space}
\def\theauthor{J. A. Cañizo, A. Gárriz \& D.A. Marín}

\hypersetup{pdftitle={\thetitle}}
\hypersetup{pdfauthor={\theauthor}}

\title{\thetitle}
\author{\theauthor}

\date{\today}

\AtEndDocument{\vspace{2cm}{\footnotesize
		\noindent
		\textsc{José A. Cañizo. Departamento de Matemática Aplicada \& IMAG,
			Universidad de Granada. Avda. Fuentenueva S/N, 18071 Granada, Spain.}
		\\
		\textit{E-mail address}: \texttt{\href{mailto:canizo@ugr.es}{canizo@ugr.es}}
		\par
		\addvspace{\medskipamount}
		\noindent
		\textsc{Alejandro Gárriz. Departamento de Matemática Aplicada \& IMAG,
			Universidad de Granada. Avda. Fuentenueva S/N, 18071 Granada, Spain.}
		\\
		\textit{E-mail address}:
		\texttt{\href{mailto:alejandro.garriz@ugr.es}{alejandro.garriz@ugr.es}}
		\par
		\addvspace{\medskipamount}
		\noindent
		\textsc{Diego Alfonso Marín. Departamento de Geometría y Topología \& IMAG,
			Universidad de Granada. Avda. Fuentenueva S/N, 18071 Granada, Spain.}
		\\
		\textit{E-mail address}: \texttt{\href{mailto:damarin@ugr.es}{damarin@ugr.es}}
}}

\begin{document}

\maketitle
	
 \begin{abstract}
   In this work we study the large-time behaviour of solutions of the Heat Equation in the hyperbolic space $\H^d$, providing precise speeds of convergence in $L^1$ and $L^\infty$ to their asymptotic profiles by means of an adaptation of entropy estimates. For $L^1$ initial conditions we are able to identify the asymptotic profile in $L^1$, which is not universal but contains a certain memory of the initial distribution of the mass of the solution. We improve thus on previous results, where speed of convergence was absent and asymptotic profiles where not known in the general case, and show a way to adapt entropy estimates employed in the study of diffusion processes to non-compact Riemannian manifolds. The main strategy to prove this is to consider  \textit{transient profiles} as minimizers of the entropy functional. These profiles are time-dependent and encompass the geometric information of the Riemannian manifold.

   %We advance that, in the case of integrable initial data, the solutions preserve along the evolution a certain \textit{memory} of the initial distribution of its mass, represented by a function on the sphere $\Phi(\theta)$, $\theta\in\S^{d-1}$.
 \end{abstract}


\renewcommand{\thefootnote}{\fnsymbol{footnote}} 
\renewcommand{\thefootnote}{\arabic{footnote}}

\tableofcontents

\section{Introduction}
\label{sec:intro}

The goal of this article is to study the behaviour for large times of positive solutions of the heat equation
\begin{equation}\label{eq:heatEquationHyperbolic}
\begin{cases}
	u_t = \Delta_{\H^d} u, \quad &x\in \mathbb{H}^d, t>0,\\
	u(0,\cdot )=u_0, & x\in \mathbb{H}^d
\end{cases}
\end{equation}
in the hyperbolic space $\mathbb{H}^d$, where $d\geq2$ represents the dimension of the space, for several different families of initial data. As we will see in this introduction, solutions to this equation present many distinctive and unexpected features that set them apart from their Euclidean counterparts and that introduce the need of several modifications on the classical techniques in order to study their large-time behaviour.

We will focus on two types of solutions of~\eqref{eq:heatEquationHyperbolic}. Those whose initial datum presents a special symmetry and, as a consequence, the solutions will depend only on one variable (as, for example, when $u_0$ is radially symmetric); and those whose initial datum only satisfies $u_0\in L^1(\H^d)$. Note that the definition of $L^1(\H^d)$ has to take into account the geometry of the space through the differential of the measure, also known as the Jacobian or the volume form, which we will denote by $d\mu$. If $u$ presents some special symmetry and depends only on one variable we will work with the classic concept of solutions of one-dimensional PDEs of second order. If, on the other hand, we work in the case where $u_0$ is integrable, then our concept of solutions will be the one constructed by the semi-group approach; a unique solution $u(t,x)$ of the heat equation with such initial data exists in that class and, moreover, $u\in C([0,\infty) : L^1(\H^d))$ and $u\in C^\infty((0,\infty)\times H^d)$, see to that respect~\cite{Grigor'yan,Strichartz}.

Specifically, the families of initial data $u_0(x)$ that we are going to work with are:
\begin{itemize}
	\item[-] Those which depend only on the distance between $x$ and a certain pole $\mathcal{O}$ and $u_0\in L^1(\R^+; \d\mu)$. We will refer to this case as \textit{radially symmetric} or \textit{radial} initial data and solutions.
	
%	\item[-] Those which depend only to the (signed) distance between $x$ and a certain hyperplane $\mathcal{P}$, which is isometric to $\H^{d-1}$, and $u_0\in L^1(\R; \d\mu)$. We will refer to this as \textit{planar}, integrable initial data and solutions.
	
	\item[-] Those which depend only on the (signed) distance between $x$ and a certain horosphere $\mathcal{H}$, and $u_0\in L^1(\R; \d\mu)$. We will refer to this as \textit{horospheric}, integrable initial data and solutions. A horosphere is an interesting geometrical object present in $\H^d$ but inexistent in $\R^d$ that will be defined in Section~\ref{sect:horospheric_coordinates}.
	
	\item[-] Those that satisfy, simply, $u_0\in L^1(\H^d; \d\mu)$. Note that integrable radial solutions are inside this family. We treat the radial solutions separately because they serve as a rich toy model and first step before treating the more general (and complex) case of only integrable solutions. Horospheric solutions are, on the other hand, clearly not in this class of integrable solutions (unless they are trivial).
\end{itemize}

To our knowledge, the first study of the behaviour for large times of the solutions of~\eqref{eq:heatEquationHyperbolic} is due to J.L. Vázquez~\cite{vazquez1992asymptotic}. In his work the author proves convergence of radial solutions to the fundamental solution of the equation $P_d(t,x)$. In particular, the author proves convergence in $L^1$ without speed, i.e., if $\mathcal{M}$ represents the mass of $u_0$ (for whatever this concept of \textit{mass} means for now), then
\[
\|u(t,x) -\mathcal{M} P_d(t,x)\|_{L^1(\H^d)}\to 0\quad\text{as}\quad t\to \infty,
\]
and, as a consequence of the $L^1-L^\infty$ smoothing effect,
\[
\|u(t,x) -\mathcal{M} P_d(t,x)\|_{L^\infty(\H^d)}= O(t^{-\frac{3}{2}} e^{-\lambda_1 t}),
\]
where $\lambda_1$ represents the first eigenvalue of the Laplace-Beltrami operator in $\H^d$, so
\begin{equation}\label{eq:definition_bottom_spectrum_laplacian}
	\lambda_1:=\frac{(d-1)^2}{4}.
\end{equation}
Other convergence results, this time for \textit{horospheric solutions} (which will be defined later), are also included in~\cite{vazquez1992asymptotic}.

But the author does not stop at these positive results. It is also proven that, if the initial datum is not radially symmetric, then the solution does not converge to the fundamental solution. If $P^y_d(t,x)$ represents the fundamental solution centred around a point $y\in\H^d$, then
\[
\lim\limits_{t\to\infty} \left[t^{\frac{3}{2}} e^{\lambda_1}|u(t,x) -\mathcal{M} P^y_d(t,x)|\right]= Q^y(x)>0\quad\text{ for all }x\in\H^d, \text{ for any } y\in\H^d,
\]
for a certain positive function $Q^y(x)$which depends also on the initial data. This result stands in sharp contrast with the Euclidean case, where there is in fact convergence to the fundamental solution centred around the centre of mass of the initial datum, as long as $u_0\in L^1(\R^d)$.

Very recently, E. Papageorgiou improved significantly on those results in~\cite{Papageorgiou2025}, following her work with Anker and Zhang in~\cite{APZ}, by providing the asymptotic profile of non-radially symmetric solutions in $L^p$, although once again without speed of convergence. Using tools from geometric harmonic analysis she proved the following. Due to the geometry of the space, the domain where most of the $L^p$ mass of the fundamental solution is located (referred as \textit{$L^p$ critical region}) varies a lot depending on whether $1\leq p<2$ (an annulus at distance $(d-1)(\frac{2}{p}-1)t$ from the origin), $p=2$ (an annulus at distance of order $\sqrt{t}$ from the origin), or $2<p\leq \infty$ (a ball centred at the origin of radius approx. $\sqrt{t}$); the reader is invited to consider the explicit formula of the kernel in dimension $d=3$ that can be found, for example, in~\cite{grigoryan1998heat}, and verify this very interesting phenomenon. This affects the asymptotics of the quotient of two heat kernels located at different points (which, as we will see, is fundamental in the study of the large-time behaviour of non-radially symmetric solutions), and as a consequence the $L^p$ convergence is modified accordingly. Using the polar coordinates for $\H^d$, where $x=(r_x, \theta_x)$, $y=(r_y, \theta_y)$ with $r_x, r_y\geq 0,\ \theta_x, \theta_y\in \S^{d-1}$ (see Section~\ref{sect:polar_coordinates} thereafter), let
\begin{equation}\label{eq:Phi_p_papagoergiou}
	\Phi_p(x) = \begin{cases}\displaystyle\int_{\H^d} u_0(y)\, \left[ \cosh (r_y)-\sinh (r_y) \cdot \langle\theta_x,\theta_y\rangle \right]^{-\frac{d-1}{p}}\ \d\mu(y)\quad&\text{when}\quad 1\leq p\leq 2,\\[20pt]
\displaystyle\frac{1}{\mathcal{S}_0(\dist(x,O))}\int_{\H^d} u_0(y) \mathcal{S}_0(\dist(x,y))\ \d\mu(y)&\text{when}\quad 2 < p,
\end{cases}
\end{equation}
where $O$ stands for the origin point, $\dist$ goes for the hyperbolic distance and $\mathcal{S}_0$ is the \textit{elementary spherical function of index 0} employed when computing the Fourier transform; its specific formula can be found in~\cite[Appendix]{Papageorgiou2025}. A sufficient condition for $\Phi_p$ to be finite is that
\begin{equation}\label{eq:condition_papgeorgiiou}
\begin{cases}\Phi_p(x)\leq \displaystyle\int_{\H^d} u_0(y)\, e^{\frac{d-1}{p}r_y}\ \d\mu(y)<\infty\quad&\text{when}\quad 1\leq p\leq 2,\\[20pt]
	\Phi_p(x)\leq \displaystyle\int_{\H^d} u_0(y)\, e^{\frac{d-1}{2}r_y}\ \d\mu(y)<\infty&\text{when}\quad 2 < p,
\end{cases}
\end{equation}
and when this is satisfied Papageorgiou proves that, for all $p\geq 1$ and for compactly supported initial data,
\begin{equation}\label{eq:convergence_papageorgiou}
(\|P_d(t,\cdot)\|_{L^p(\H^d)})^{-1}\cdot \|u(t,x)-\Phi_p(x)P_d(t,x)\|_{L^p(\H^d)}\to 0\quad\text{as}\quad t\to\infty.
\end{equation}
We highlight that in the case $1\leq p\leq 2$, if we consider polar coordinates, the function $\Phi_p$ depends only on the spherical coordinate $\theta_x$, and that condition~\eqref{eq:condition_papgeorgiiou}, while sufficient, may not be necessary for~\eqref{eq:convergence_papageorgiou}, as we will see. Also, most notably, Papageorgiou proves such convergence results not only on the Hyperbolic space, but on
arbitrary-rank Riemannian symmetric spaces of non-compact-type $G/K$ for (possibly) non-bi-$K$ invariant initial data (which means, in our particular case, not radially symmetric data). 

\begin{rem}
	%When developing our study, we arrived to the same results as Papageorgiou independently in the case $p=1$. It was later that we learnt about her work~\cite{Papageorgiou2025}, which offered much insight on this subject, and allowed us to shorten greatly this article. However, w
	We find that the techniques coming from geometric harmonic analysis are somewhat distant to those of PDEs, and as such we dedicate some pages of this work to briefly introduce independent techniques that provide similar results to hers and complement them, with a language more natural to readers with a background on PDEs.
\end{rem}

In our present work we improve on the aforementioned results by providing a speed of convergence of the $L^1$ norm of the difference between $u$ and its asymptotic profile of the form
\[
\|u(t,x)-F(t,x)\|_{L^1(\H^d)}\leq h(t), \quad\text{where}\quad h(t)\to 0 \text{ as } t\to\infty
\]
for certain asymptotic profile $F$ and convergence speed $h$, for a more general class of initial data than~\cite{Papageorgiou2025}. In the case of radial solutions this will imply also a speed of convergence in $L^\infty$ by means of the $L^1-L^\infty$ smoothing effect and in all other $L^p$ norms by interpolation; this direct improvement does not apply in the general case of non-radial solutions. We will also improve those results by obtaining a speed of convergence of horospheric solutions (covered partially in~\cite{vazquez1992asymptotic} and not included in ~\cite{Papageorgiou2025}). Finally, our approach has the advantage of deriving our results with PDEs tools, avoiding most of the geommetric harmonic analysis techniques employed by Anker, Papageorgiou and Zhang, presenting all of them under a unified framework, with a strong enough technique to tackle all the cases (radial, horosopheric or integrable solutions) similarly. 

The aforementioned unified framework is that of entropy estimates. We will employ a refinement of these ideas, of which an early presentation can be found in~\cite{Arnold-Markovich-Toscani-Uterretiter-2001,toscani,Toscani99}, adapted to work in $\H^d$, a Riemannian manifold whose particular geometry makes it so the classical theory does not apply directly; up to our knowledge, this is the first time that such ideas are employed in non-compact Riemannian manifolds with non-trivial curvature. As we will see, the Riemannian metric of the space (and, in particular, the differential of the measure of this space) will play a fundamental role in our computations. Since some previous definitions about the coordinate system for $\H^d$ are needed, we postpone the introduction of these ideas to later in this section. For an introductory bibliography on entropy estimates applied to diffusion processes we mention~\cite{Arnold-Carrillo-Desvillettes-et-al-2004,Vazquez-Arxiv-2018}. This technique has been employed fruitfully to other diffusion processes, like the Porous Medium Equation~\cite{Carrillo-Toscani-2000}, the Fast Diffusion Equation~\cite{Bonforte-Dolbeault-Grillo-Vazquez-2010}, the Fractional Heat Equation~\cite{Gentil-Imbert-2008}, or the Heat Equation in exterior domains~\cite{CGQ}, just to name a few.

Since this article is, in principle, oriented to readers closer to the world of PDEs, we will devote some time to introduce, briefly, the geometric concepts needed in this work.

\subsection{Introduction to the Geometry of the space $\H^d$}

To describe points \(x \in \mathbb{H}^d\), one first needs to introduce a coordinate system. Unlike in Euclidean space, there is no universally accepted “canonical” coordinate system in \(\mathbb{H}^d\); instead, several different models are available to represent hyperbolic spaces. In this work, we focus on coordinate systems that are best suited to capture the geometric symmetries inherent in \(\mathbb{H}^d\). 

We say that an isometry of $\mathbb{H}^d$ is \textit{elliptic} when it fix all geodesic spheres centered at some point, and we say that it is \textit{parabolic} when it fix a special class of hypersurfaces of $\mathbb{H}^d$, called \textit{horospheres}. To deal with functions which are invariant under one of these two groups of isometries we use coordinates in which one of the parameters is the distance to a given sphere or horosphere. These kind of coordinates are a special case of a general construction, called \textit{Fermi coordinates}. In Section~\ref{sect:construction_coordinates} of the appendix we provide an introduction of the construction of such coordinate systems for the interested readers. We now present the two coordinate systems we will use along this paper. 

\subsubsection{Polar coordinates, also called elliptic coordinates}\label{sect:polar_coordinates}

We can describe the group of elliptic isometries as the subgroup of $\textup{Iso} (\mathbb{H}^d)$ whose principal orbits are given by geodesic spheres centred at some point.

Given $o \in \mathbb{H}^d$ and $r_0 >0$, we define the geodesic ball centred in $o$ and with radius $r_0>0$ as the set $B(o, r_0 )= \{x \in \mathbb{H}^d ~\colon~ \textup{dist}(o,x)< r_0\}$, and the corresponding geodesic sphere as $S(o, r_0 ):= \partial B(o, r_0 )$. Take $N$ to be the normal vector field on $S(o, r_0 )$ that points outside $B(o, r_0 )$. Then,  since each geodesic sphere is diffeomorphic to $\mathbb{S}^{d-1}$, by taking a coordinate neighbourhood centred at $S(o, r_0 )$ with respect to $N$, we get that $\mathbb{H}^d \setminus \{o\}$ is diffeomorphic to $(-r_0, +\infty) \times \mathbb{S}^{d-1}$ and the hyperbolic metric takes the form
\[
g_{\mathbb{H}^d}= \d r^2 + \sinh^2(r+r_0)g_{\mathbb{S}^{d-1}}
\]
in these coordinates, where $g_{\mathbb{S}^{d-1}}$ is the round metric in the sphere $\mathbb{S}^{d-1}$. 

By taking $r_0 \to 0$ we get the so called \textit{polar coordinates} in $\mathbb{H}^d$, which have the following interpretation: for any $x \neq o$, we describe the point by its distance to $o$, $r= \textup{dist}_{\mathbb{H}^d} (o,x)$, and the direction $\theta \in \mathbb{S}^{d-1} \subset \mathbb{R}^d \equiv T_o \mathbb{H}^{d-1}$ such that $\textup{exp}_o (r \theta) =x$. In this case, $\mathbb{H}^d \setminus \{ o \}$ is diffeomorphic to $(0,+\infty) \times \mathbb{S}^{d-1}$ and the metric is given by
\[
g_{\mathbb{H}^d}= \d r^2 + \sinh^2(r)g_{\mathbb{S}^{d-1}}
\]
in these coordinates. Also, the Laplace-Beltrami operator reads
\begin{equation}\label{eq:laplace_beltrami_polar_coordinates}
	\Delta_{\H^d} u = \frac{1}{\sinh^{d-1}(r)}\partial_r\left(\sinh^{d-1}(r)\cdot \partial_r u\right) + \frac{1}{\sinh^2(r)}\Delta_{\mathbb{S}^{d-1}}u
\end{equation}
and the volume form is given by $\d\mu(r,\theta)=\sinh^{d-1}(r)\d r \d \theta$, where $\d \theta$ denotes the volume form of the standard round metric in $\mathbb{S}^{d-1}$. 

%\subsubsection{Hyperbolic coordinates}
%Let \(o \in \mathbb{H}^d\) be a fixed point, and let \(P \subset T_o \mathbb{H}^d\) be a hyperplane in the tangent space. We define the subset $\mathcal{P} := \exp_o(P)$ to be a \textit{totally geodesic hypersurface}. Such hypersurfaces arise as the principal orbits of the group of hyperbolic isometries of \(\mathbb{H}^d\).  
%
%Fix a normal field $N$ to $\mathcal{P}$. Then, by introducing a coordinate neighborhood centered at \(\mathcal{P} \subset \mathbb{H}^d\), one obtains a diffeomorphism $\mathbb{H}^d \cong \mathbb{R} \times \mathbb{R}^{d-1}$.
%When working in this coordinate neighborhood, we shall refer to it as employing \textit{hyperbolic coordinates}. Since every totally geodesic hypersurface is isometric to \(\mathbb{H}^{d-1}\), the hyperbolic metric in these coordinates takes the form  
%\[
%g_{\mathbb{H}^d} = \d r^2 + \cosh^2(r)\, g_{\mathbb{H}^{d-1}}
%\]
%and, consequently, the Laplace--Beltrami operator is given by  
%\begin{equation}\label{eq:laplace_beltrami_hyperbolic_coordinates}
%	\Delta_{\mathbb{H}^d} u \;=\; \frac{1}{\cosh^{d-1}(r)}\,\partial_r\!\Big(\cosh^{d-1}(r)\,\partial_r u\Big) \;+\; \frac{1}{\cosh^2(r)}\,\Delta_{\mathbb{H}^{d-1}} u.
%\end{equation}
%With respect to this metric, the volume form reads $\d \mu(r,w) = \cosh^{d-1}(r)\,\d r\,\d w$, where \(\d w\) denotes the volume element of \(\mathbb{H}^{d-1}\).  

\subsubsection{Parabolic coordinates, also called horospheric coordinates}\label{sect:horospheric_coordinates}

Let $o \in \mathbb{H}^d$ and take $\gamma : [0, +\infty) \to \mathbb{H}^d$ a geodesic ray starting at $o$ and parametrised  by the arc length, meaning that $\textup{dist} (o, \gamma (t)) =t$ for any $t \in [0,+\infty)$. Then the set given by the limit
\[
\mathcal{H}B = \displaystyle\lim_{t \to +\infty} B(\gamma (t), t)
\]
is an open set of $\mathbb{H}^d$ called a \textit{horoball}. Its boundary $\mathcal{H} = \partial \mathcal{H}B$ is a smooth hypersurface which is referred to as a horosphere. Horospheres are invariant under the group of parabolic isometries of $\mathbb{H}^d$, and can be thought as spheres centered at some point of the ideal boundary of $\mathbb{H}^d$ (see for example \cite{petersen2016riemannian} for a description of this set).

Fix a horosphere \(\mathcal{H} \subset \mathbb{H}^d\), and denote by \(\mathcal{H}B\) 
the horoball it bounds. Let \(N\) be the unit normal vector field pointing toward 
the complement \(\mathbb{H}^d \setminus \mathcal{H}B\). With this convention, 
the signed distance to \(\mathcal{H}\) is positive outside \(\mathcal{H}B\) 
and negative inside. Since every horosphere is isometric to the euclidean space 
\(\mathbb{R}^{d-1}\), the corresponding Fermi coordinate neighborhood provides 
a parametrization of \(\mathbb{H}^d\) by \(\mathbb{R} \times \mathbb{R}^{d-1}\), 
in which the hyperbolic metric takes the form 
\[
g_{\mathbb{H}^d}= \d r^2 + e^{2 r} g_{\mathbb{R}^{d-1}}.
\]
With respect to this coordinate system, the volume form of \(\mathbb{H}^d\) is $\d\mu(r,y) = e^{(d-1)r} \, \d r \, \d y,$ where \(\d y\) denotes the Euclidean volume form on \(\mathbb{R}^{d-1}\). Moreover, 
for a function \(u\), the Laplace--Beltrami operator is given by  
\begin{equation}\label{laplacianParabolic}
	\Delta_{\mathbb{H}^d} u 
	= \frac{1}{e^{(d-1)r}} \, \partial_r \!\left(e^{(d-1)r} \cdot \partial_r u\right) 
	+ \frac{1}{e^{2r}} \, \Delta_{\mathbb{R}^{d-1}} u.
\end{equation}

\begin{rem}
	There is a third possible coordinate system corresponding to the hyperbolic subgroup of $\text{Iso}(\H^d)$ that we do not cover in this article. Let $\mathcal{P}\subset\H^d$ be a hyperplane. Then any $x\in \H^d$ can be expressed as a point $w\in \mathcal{P}$ and a signed distance $r\in\R$ between $w$ and its projection over $\mathcal{P}$, so $x=(w, r)$. This is often referred as employing \textit{hyperbolic coordinates}. Since every totally geodesic hypersurface is isometric to \(\mathbb{H}^{d-1}\), the hyperbolic metric in these coordinates takes the form  
	\[
	g_{\mathbb{H}^d} = \d r^2 + \cosh^2(r)\, g_{\mathbb{H}^{d-1}}.
	\]
	The ideas employed in this work are applicable, \textit{mutatis mutandis}, to solutions that depend only on this distance $r$ to a certain hyperplane $\mathcal{P}$; we can refer to such solutions as \textit{planar} solutions. However, many more computations and auxiliary results are needed in order to study this case, and as a consequence we decided to exclude \textit{planar} solutions from the present article and leave them for future work.
\end{rem}

\subsection{The heat kernel in hyperbolic space}
Given any $y\in \mathbb{H}^d$, we will denote as $P^y_d(t,x)$ the fundamental solution of equation~\eqref{eq:heatEquationHyperbolic} at time $t>0$ and point $x\in \mathbb{H}^d$ with initial datum $u_0(x)= \delta_y(x)$, Dirac's Delta function centred at the point $y$. As in the euclidean case, any solution with integrable initial datum can be written as
\[
u(t,x)=\int_{\mathbb{H}^d} u_0(y)P^y_d(t,x) \d\mu(y)
\]
It is easy to verify a solution of the heat equation preserves its mass, i.e.,
\[
\|u(t,\cdot)\|_{L^1(\mathbb{H}^d)} = \|u_0\|_{L^1(\mathbb{H}^d)}=\int_{\mathbb{H}^{d}} u_0(y) \d\mu(y).
\]
Since the value of $P^y_d(t,x)$ depends only on time and the distance between $x$ and $y$, we will write
\[
P^y_d(t,x) = G_d(t,\lambda(x,y)),\quad\text{where}\quad \lambda(x,y)=\text{dist}(x,y)
\]
for a certain function $G_d$ whose explicit formula can be found in \cite{Davies_1989, Grigor'yan}. For any $x\in \mathbb{H}^d$, write $x \equiv (r_x, \theta_x)$ for its polar coordinates with respect to a fixed origin $o \in \mathbb{H}^d$, where we recall that $r_x=\dist(x,o)\geq 0$ and $\theta_x\in \mathbb{S}^{d-1}$. Then, by the second hyperbolic law of the cosines, we have
\begin{equation}\label{eq:formula_lambda}
	\lambda(x,y)= \text{arccosh}(\cosh(r_x)\cosh(r_y)-\sinh(r_x)\sinh(r_y)\cdot \langle\theta_x,\theta_y\rangle),
\end{equation}
where $\langle \cdot ,\cdot \rangle$ denotes the euclidean scalar product, and thus $\langle\theta_x,\theta_y\rangle$ is the cosine of the angle formed by the geodesics connecting the points $x$ and $y$ to the origin. If we call such angle $\vartheta$, then we have
\[
\vartheta(\theta_x,\theta_y) = \arccos(\langle\theta_x,\theta_y\rangle)
\]
and then the formula for $\lambda$ becomes
\[
\lambda(x,y)= \text{arccosh}(\cosh(r_x)\cosh(r_y)-\sinh(r_x)\sinh(r_y)\cdot \cos(\vartheta)),
\]
writing $\vartheta$ instead of $\vartheta(\theta_x,\theta_y)$ for simplicity.

The formula for the solution $u$ becomes then
%\begin{equation}\label{eq:solution_as_convolution_with_kernel}
\[
	u(t,x)= u(t,r_x,\theta_x) = \int_{\mathbb{S}^{d-1}}\int_0^\infty u_0(r_y,\theta_y)\cdot G_d(t,\lambda(x,y))\cdot  \sinh^{d-1}(r_y)\d r_y \d \theta_y.
\]
%\end{equation}
There are, moreover, two notable relations regarding the function $G_d$, see for example \cite{grigoryan1998heat}. First, if we define
\[
h_d(t,r):= (4\pi t)^{-\frac{d}{2}}(1+r+t)^{\frac{d-3}{2}}(1+r)e^{-\frac{(r+(d-1)t)^2}{4t}}
\]
then there exists a couple of positive constants $c_d, C_d$ such that, for all $t>0, r>0$,
\begin{equation}\label{eq:bound_heat_kernel_in_terms_of_h}
	c_d\ h_d(t,r)\leq G_d(t,r)\leq C_d\ h_d(t,r),
\end{equation}
which improves to the complete asymptotics
\[
	G_d(t,r)\sim \gamma\left(\frac{r}{2t}\right)t^{-\frac{3}{2}}\ r\ e^{-\frac{(r+(d-1)t)^2}{4t}}
\]
as $t\to\infty$ and $r\to \infty$, where $\gamma$ is a known function depending on the Gamma-function, see~\cite{Anker-Ji-1999} and~\cite[Appendix]{Papageorgiou2025} for the details.

Second, the derivative of $G_d$ satisfies the formula
\begin{equation}\label{eq:recurrence_derivative_kernels}
	\partial_r G_d(t,r) = -2\pi e^{dt}\sinh(r) G_{d+2}(t,r).
\end{equation}

\subsection{Introduction to the Entropy Estimates}

Let us discuss briefly the ideas behind the entropy estimates in polar coordinates. From the work of \cite{vazquez1992asymptotic} (see also \cite[Appendix A]{LM}) we see that the bulk of the mass of the heat kernel in $\H^d$ centred around a point $\mathcal{O}$ is located, at points $x=(r,\theta)$ and at time $t$, at a distance $r=\text{dist}(\mathcal{O}, x)= (d-1)t$. Based on this we define the change of variables
\[
t(\tau)=e^\tau - 1,\quad r(\tau, \rho) = (d-1)e^\tau + e^\frac{\tau}{2}\rho,
\]
and, given a solution $u$ of~\eqref{eq:heatEquationHyperbolic}, consider the auxiliary function
\[
v(\tau, \rho, \theta) = e^\frac{\tau}{2}\cdot \sinh^{d-1}(r(\tau, \rho))\cdot u(t(\tau), r(\tau, \rho), \theta).
\]
This function $v$ satisfies a certain parabolic equation. Now, we consider the adapted Kullback-Leibler entropy functional
\[
H(\tau):= \int_{\mathbb{S}^{d-1}}\int_{\rho_0(\tau)}^\infty v\cdot\ln\left(\frac{v}{\mathcal{F}}\right)\ \d \rho\ \d \theta,
\]
for a certain profile $\mathcal{F}(\tau, \rho, \theta)$ that will play the role of the \textit{equilibrium} towards $v$ is converging; this is called the entropy of $v$ relative to $\mathcal{F}$, or relative entropy, for brevity. The specific definition of $\mathcal{F}$ depends on the characteristics of the problem, and will not be the same for radial, horospheric or integrable solutions. There are two noteworthy deviations from the classical entropy techniques in this case. First, there is an unusual dependence on $\tau$ both on the limits of the integral and on the profile $\mathcal{F}$ (called for this reason a \textit{transient equilibrium}); nonetheless, the method holds with some extra steps, see \cite{CGQ}. Second, there is an unexpected factor $\sinh^{d-1}(r)$ appearing in the definition of $v$; it corresponds to the volume form $\d\mu = \sinh^{d-1}(r) \d r \d  \theta$. By including it here we bring, in the sense of integrals, the function $v$ from the geometry of $\H^d$ into the geometry of $\R^d$, where computations can be carried out. It is in this sense that we claim to adapt the entropy techniques to the frame of non-compact Riemannian manifolds.

Via some analysis, the goal is to prove a decay of the form
\[
H(\tau)\leq (C+H(0))e^{-\lambda \tau}
\]
for some positive constants $C, \lambda$. We often write indistinctively $H(\tau)$ or $H(v)$ depending on whether we want to emphasize the dependence on the function $v$ or on the time $\tau$, and also $H(u_0), H(v_0)$ or $H(0)$ depending on whether we want to emphasize the dependence on the initial datum $u_0$, $v_0=v(0, \rho, \theta)$ or on the time $\tau=0$, but we are referring to the same quantities; the context should make it clear. With this decay at hand, the Csiszár-Kullback inequality yields
\[
\|v-\mathcal{F}\|_{L^1}\leq \sqrt{C+H(0)}\cdot e^{-\frac{\lambda}{2}\tau},
\]
and then, undoing the change of variables, we obtain convergence of $u$ to a certain profile $\mathcal{F}^*$ in the $L^1(\d\mu)$ norm. In some cases, an $L^1-L^\infty$ regularisation effect will even provide convergence in the $L^\infty$ norm.

\subsection{Main results}

According to each of the two coordinate systems that we have presented there are symmetries on the initial datum that fits them, and these symmetries affect, of course, the large-time behaviour of the solutions. First, for the initial data that depends only on the distance to a certain point $\mathcal{O}\in\H^d$, the natural coordinates are the polar ones with pole at $\mathcal{O}$, where $x=(r_x,\theta_x)$, with $r_x\geq 0, \theta_x\in\S^{d-1}$. We recall that we write the heat kernel as
\[
P_d(t, x) = G_d(t, r_x),
\]
since it is a radial function centred around the pole $\mathcal{O}$ that depends only on time and the distance $r_x$, where we write, for simplicity of notation, $P_d$ for the fundamental solution centred at the origin $\mathcal{O}$. We recall the definition of $\lambda_1$ as the bottom of the spectrum of the Laplace-Beltrami operator given in~\eqref{eq:definition_bottom_spectrum_laplacian}.

The following result is a consequence of Thorems~\ref{thm:L^1_radially_symmetric_asymptotic} and~\ref{thm:radially_symmetric_L^infty asymtotic}, in Section~\ref{sect:elliptic}.
\begin{thm}[Radially symmetric initial data]
	Let the polar coordinates be such that a point $x\in\H^d$ is given by $x=(r_x, \theta_x)$, where $r_x\geq 0$ represents the hyperbolic distance between $x$ and the pole $\mathcal{O}$ and $\theta_x\in\S^{d-1}$ represents the direction of the unit vector from $\mathcal{O}$ to $x$. Let $u_0\in L^1(\H^d)$ such that $u_0\geq 0$ and $u_0(x)=u_0(r_x)$ for all $x\in\H^d$. Then $u(t,x)=u(t, r_x)$ for all $t\geq 0$ and there exists a $t_0$ such that for all $t\geq t_0$,
	\[
	\begin{aligned}
	\| u(t,\cdot) - \mathcal{M}_{\mathcal{R}}P_d(t,\cdot)  \|_{L^1(\mathbb{H}^d)}&=\| u(t,r_x) - \mathcal{M}_{\mathcal{R}}G_d(t,r_x)  \|_{L^1(\mathbb{H}^d)}\\[10pt]
	&\lesssim \left(\sqrt{H^*(G_d)}+\sqrt{H^*(u_0)}\right) \cdot t^{-\frac{1}{2}},
	\end{aligned}
	\]
	where $\mathcal{M}_{\mathcal{R}}=\|u_0\|_{L^1(\H^d)}$ and $H^*(u_0)= K(d, \mathcal{M}_{\mathcal{R}}) + H(u_0)$, for some constant $K$ that depends only on the dimension $d$ and the mass $\mathcal{M}_{\mathcal{R}}$. Moreover, as a consequence of the $L^1-L^\infty$ regularizing effect,
	\[
	\begin{aligned}
		\| u(t,\cdot) - \mathcal{M}_{\mathcal{R}}P_d(t,\cdot)  \|_{L^\infty(\mathbb{H}^d)}\lesssim \left(\sqrt{H^*(G_d)}+\sqrt{H^*(u_0)}\right) \cdot e^{-\lambda_1 t}t^{-2}.
	\end{aligned}
	\]
\end{thm}  

%In the case where the initial datum depends only on the distance to a hyperplane $\mathcal{P}\subset\H^d$, the we employ hyperbolic coordinates, where $x=(r,w)$, with $r\in\R, w\in\H^{d-1}$. Again, the solution will only depend on the parameter $r$, ad we have the following.
%
%\begin{thm}
%	Let $u_0\in L^1(\H^d)$ such that $u_0(x)=u_0(r)$ for all $x\in\H^d$. Then $u(t,x)=u(t, r)$ for all $t\geq 0$ and
%	\[
%	\ag{TO COMPLETE}
%	\]
%\end{thm}

Next, when the initial datum depends only on the distance to a horosphere $\mathcal{H}$, then the same can be said about the solution, so we choose parabolic coordinates $x=(r_x,y_x)$, with $r_x\in\R, y_x\in\mathcal{H}$ and write $\d\nu(r)=e^{(d-1)r}\d r$ for the horospherical part of the volume form of $\mathbb{H}^d$. We obtain the following theorem as a consequence of Theorems~\ref{thm:convregece_horoshepric_L^1} and~\ref{thm:convregece_horoshepric_L^infty}, in Section~\ref{sect:horospheric}.

\begin{thm}[Horospheric initial data]
	Let the horospheric coordinates be such that a point $x\in\H^d$ is given by $x=(r_x, y_x)$, where $r_x\in\R$ represents the signed hyperbolic distance between $x$ and the horosphere $\mathcal{H}$ and $y_x\in\mathcal{H}$ represents the projection of $x$ onto $\mathcal{H}$. Let $u_0\in L^1(\R, \d\mu)$ such that $u_0\geq 0$ and $u_0(x)=u_0(r_x)$ for all $x\in\H^d$. Then $u(t,x)=u(t, r_x)$ for all $t\geq 0$ and
%	\[
%	\displaystyle\int_{\R} \abs{u(t,r)- \mathcal{M}_{\mathcal{H}}\ \Gamma(t, r+(d-1)t)} e^{(d-1)r} \, \d r \lesssim \sqrt{H(u_0)} \cdot t^{-\frac{1}{2}}
%	\]
	\[
	\|u(t,r)- \mathcal{M}_{\mathcal{H}}\ \Gamma(t, r+(d-1)t) \|_{L^1(\R, \d\nu)} \lesssim \sqrt{H(u_0)} \cdot t^{-\frac{1}{2}}
	\]
	where $\Gamma(t,r)$ stands for the Gaussian profile
	\[
	\Gamma(t,r)=\frac{1}{\sqrt{t}}e^{-\frac{r_x^2}{4t}}
	\]
	and $\mathcal{M}_{\mathcal{H}} = \|u_0\|_{L^1(\mathbb{R}; \d\nu)}$. Moreover,
	\[
	\|u(t,r)- \mathcal{M}_{\mathcal{H}}\ \Gamma(t, r+(d-1)t) \|_{L^\infty(\R)} \lesssim \sqrt{H(u_0)} \cdot t^{-1}.
	\]
\end{thm}

However, these results leave behind many physically relevant cases where the initial datum does not have a particular symmetry to exploit. The main result of this article deals with this case. From the two coordinate systems presented in this article, the one best suited to describe solutions whose initial datum satisfies only $u_0\in L^1(\H^d)$ is the polar coordinates system, where $x=(r_x, \theta_x)\in\H^d$ with $r_x\geq 0$ and $\theta_x\in \S^{d-1}$. Naturally, in this setting, $u_0\in L^1(\H^d)$ means $u_0\in L^1(\R^+\times\S^{d-1}; \d\mu)$ with $\mu(r, \theta) = \sinh^{d-1}(r)$ 

First, we need to introduce the ``memory'' function $\Phi(\theta_x)$ depending only on the initial datum $u_0$. The reader can think of it as the amount of mass of $u_0$ in the direction $\theta_x$ weighted accordingly to the function
\[
\varphi(y,\theta_x)=[\cosh(r_y)-\sinh(r_y)\cdot\left\langle\theta_x,\theta_y\right\rangle]^{-(d-1)},
\]
so
\[
\Phi(\theta_x) = \int_{\H^d} u_0(y)\, \varphi(y, \theta_x)\ \d\mu(y).
\]
For this memory function, we assume $\Phi(\theta_x)\in L^1(\S^{d-1})$. This assumption is enough to obtain convergence, see Theorem~\ref{thm:convergence_without_rate}. Notice that $\Phi$ matches $\Phi_1$ in~\eqref{eq:Phi_p_papagoergiou}, as it should. We arrive however at this quantity independently from~\cite{Papageorgiou2025}.

Second, in order to improve Theorem~\ref{thm:convergence_without_rate} by means of a speed of convergence, the initial datum has to satisfy a certain growth condition, namely
\begin{equation}\label{eq:growth_condition_introduction}
\sup\limits_{\theta_x\in\S^{d-1}}\left\{\int_C^\infty u_0(\theta_x, r)\sinh^{d-1+\varepsilon}(r)\, \d r\right\} <\infty
\end{equation}
for some fixed positive constants $C,\varepsilon$. This can be thought of as a sort of ``momentum'' of order $\varepsilon$; of course, since $u_0\in L^1(\H^d)$ one must think about $\varepsilon << 1$. The specifics can be found in Lemma~\ref{lem:growth_condition_u_0}. Notice that if $u_0$ is compactly supported then condition~\eqref{eq:growth_condition_introduction} is trivially satisfied.

With this at hand, we can write down the main result of this article. We define for now $H^*(u_0)$ as a value depending only on the initial datum, the initial entropy $H(u_0)$ and the dimension $d$ such that
\[
H(u_0)<H^*(u_0),\quad\text{and}\quad H^*(u_0)<\infty\  \text{ if }\  H(u_0)<\infty.
\]
As a consequence of Theorem~\ref{thm:integrable_L^1_Lînfty_convergence} in Section~\ref{sect:convergence_with_rates} we obtain the following.

\begin{thm}[$L^1$ initial data]
	Let the polar coordinates be such that a point $x\in\H^d$ is given by $x=(r_x, \theta_x)$, where $r_x\geq 0$ represents the hyperbolic distance between $x$ and the pole $\mathcal{O}$ and $\theta_x\in\S^{d-1}$ represents the direction of the unit vector from $\mathcal{O}$ to $x$. Let $u_0\in L^1(\H^d)$ such that $u_0\geq 0$, $\Phi(\theta_x)\in L^1(\S^{d-1})$. Then,
	\[
	\| u(t,x) - C_d\ \Phi(\theta_x)\ \Gamma(t, r_x+(d-1)t) \|_{L^1(\mathbb{H}^d)}\to 0\quad\text{as}\quad t\to \infty,
	\]
	where
	\[
	C_d=\frac{2^{d-2}}{\sqrt{\pi}\cdot |\S^{d-1}|},
	\]
	or, in terms of the heat kernel,
	\[
	\| u(t,\cdot) - \Phi(\cdot)\ P_d(t,\cdot) \|_{L^1(\mathbb{H}^d)}=\| u(t,x) - \Phi(\theta_x)\ G_d(t,r_x) \|_{L^1(\mathbb{H}^d)}\to 0\text{ as } t\to \infty.
	\]
	If, in addition, condition~\eqref{eq:growth_condition_introduction} is satisfied, then there exists a time $t_0$ such that, for all $t\geq t_0$,
	\[
	\| u(t,x) - C_d\ \Phi(\theta_x)\ \Gamma(t, r_x+(d-1)t) \|_{L^1(\mathbb{H}^d)}  \lesssim \sqrt{H^*(u_0)}\cdot t^{-\frac{1}{2}}.
	\]
	Equivalently, in terms of the heat kernel,
	\[
	\begin{aligned}
	\| u(t,\cdot) - \Phi(\cdot)\ P_d(t,\cdot) \|_{L^1(\mathbb{H}^d)}&=\| u(t,x) - \Phi(\theta_x)\ G_d(t,r_x) \|_{L^1(\mathbb{H}^d)}\\[10pt]
	&\lesssim \sqrt{H^*(u_0)}\cdot t^{-\frac{1}{2}}\quad \text{for all}\quad t\geq t_0.
	\end{aligned}
	\]
\end{thm}

\subsection{Organisation of the article}

Section~\ref{sect:convergence_symmetries} is devoted to the study of the asymptotic behaviour of solutions presenting a certain symmetry, either radial (Section~\ref{sect:elliptic}) or horospheric (Section~\ref{sect:horospheric}); it can be regarded as the simpler cases, where the entropy estimates are introduced. Section~\ref{sect:convergence_integrable_no_rates} deals with the general case where $u_0$ is nonnegative and integrable, introducing the ``memory function'' $\Phi(\theta)$ and finding the asymptotic profile of the solutions but without a rate of convergence. It is in Section~\ref{sect:convergence_integrable_with_rates} where we obtain such rates via entropy estimates and a deep study of the directional mass distribution of the solution (Section~\ref{sect:directional_mass_distribution}). 

In the Appendix we collect several computations and auxiliary results that, for the sake of clarity, have been separated from the main text. Section~\ref{sect:construction_coordinates} is a brief introduction to the derivation of coordinates systems for Riemannian manifolds for the interested readers. Section~\ref{sect:sobolev} is an addendum in which we present the auxiliary one-dimensional logarithmic Sobolev inequalities needed in this article to obtain the rates of convergence of the entropies. Section~\ref{sec:UsefulBounds} is devoted to a useful bound for the radially symmetric equilibrium, and Section~\ref{sect:ExplicitVarPhi} provides, as an example, the computations on how to compute function \( \varphi \), defined in Section~\ref{sect:convergence_integrable_no_rates}, in the case $d\in\{2,3\}$ just with some analysis.

%Finally, in Section~\ref{sec:L1Data} we present a generalization of the results from Section~\ref{sect:convergence_integrable_with_rates} to the case of non-compactly supported initial data.

\section{Convergence of symmetric solutions via entropy estimates}\label{sect:convergence_symmetries}
In this section we deal with solutions to \eqref{eq:heatEquationHyperbolic} that present one of the symmetries mentioned in the introduction. The results presented here differ mainly from those contained in \cite{vazquez1992asymptotic} and \cite{APZ} in the fact that here we obtain convergence results with speed of convergence.

\subsection{Radially symmetric solutions. }\label{sect:elliptic}

Given the operator~\eqref{eq:laplace_beltrami_polar_coordinates}, if we consider initial data that is radially symmetric then the solution stays radially symmetric, see \cite{vazquez1992asymptotic}, and we can omit the spherical part, writing with a small abuse of notation
\[
u(t, x)= u(t,r)\quad\text{for any}\quad \theta \in\mathbb{S}^{d-1},\quad \Delta_{\H^d} u = \frac{1}{\sinh^{d-1}(r)}\partial_r\left(\sinh^{d-1}(r)\cdot \partial_r u\right).
\]
Equation~\eqref{eq:heatEquationHyperbolic} then becomes one-dimensional and reads
\begin{equation}\label{eq:heatEquation_radial_solutions}
	\begin{cases}
		u_t = \frac{1}{\sinh^{d-1}(r)}\partial_r\left(\sinh^{d-1}(r)\cdot \partial_r u\right), \quad &r>0, t>0,\\
		u(0,r)=u_0(r), & r>0.
	\end{cases}
\end{equation}
Since the mass is preserved, let us define then, for the rest of this section,
\[
\mathcal{M}_{\mathcal{R}}:=\int_0^\infty u_0(r)\sinh^{d-1}(r)\ \d r,
\]
implying that
\[
\|u_0\|_{L^1(\mathbb{H}^d)}=|\S^{d-1}|\cdot \mathcal{M}_{\mathcal{R}}
\]
Now, since we know that the bulk of mass of the solution is contained within a $\sqrt{t}$-neighborhood of the geodesic sphere $\{r= (d-1)t\}$ (see \cite{vazquez1992asymptotic}), it make sense to consider the change of variables
\[
v(\tau, \rho):=e^\frac{\tau}{2}\cdot\sinh^{d-1}\big(r(\tau,\rho)\big)\cdot u\big(t(\tau), r(\tau,\rho)\big),
\]
where
\[
t(\tau)=e^\tau-1,\quad r(\tau,\rho)= (d-1)e^\tau  + e^\frac{\tau}{2}\rho.
\]
In this way, we locate our coordinates around the mass of our solution $u$ to the heat equation. The factor $\sinh^{d-1}\big(r(\tau,\rho)\big)$ is included to work with the function in $L^1(\mathbb{R},dr)$, instead of $L^1(\mathbb{R},\mu (r)dr)$.

Since $r\geq 0$ we must have
\[%\begin{equation}\label{eq:def_rho_0_tau}
\rho\geq \rho_0(\tau):= -(d-1)e^\frac{\tau}{2}
%\end{equation}
\]
and thus
\[
v(\tau, \rho_0(\tau)) = e^\frac{\tau}{2}\cdot\sinh^{d-1}\big(\underbrace{r(\tau,\rho_0(\tau)}_{=0})\big)\cdot u\big(t(\tau), \underbrace{r(\tau,\rho_0(\tau)}_{=0}\big) = 0
\]
This function $v$ also preserves the mass in its domain of definition, i.e.,
\[
\int_{\rho_0(\tau)}^\infty v(\tau, \rho)\ \d \rho = \int_0^\infty v_0(\rho)\ \d \rho = \int_0^\infty u_0(r)\sinh^{d-1}(r)\ \d r = \mathcal{M}_{R}\quad\text{for all}\quad \tau\geq 0.
\]
The previous definitions provide
\begin{equation}\label{derivatives}
r_\tau = (d-1)e^\tau  + \frac{1}{2}e^\frac{\tau}{2}\rho,\quad r_\rho = e^\frac{\tau}{2}, \quad \partial_r = e^{-\frac{\tau}{2}}\cdot \partial_\rho,
\end{equation}
and, as a consequence of the last relation,
\[
u_r = e^{-\frac{\tau}{2}}u_\rho.
\]
With the help of~\eqref{derivatives}, we show that $v$ solves the equation
%\[
%\begin{aligned}
%v_\tau &= \frac{1}{2}v + (d-1)e^\frac{\tau}{2}\ \sinh^{d-2}(r)\cosh(r)\ r_\tau\ u +e^\frac{3\tau}{2}\ \sinh^{d-1}(r)\ u_t + e^\frac{\tau}{2}\ \sinh^{d-1}(r)\ r_\tau\ u_r\\
%&= \frac{1}{2}v + (d-1)\coth(r)\ r_\tau\ v +e^\frac{3\tau}{2}\ \partial_r\left(\sinh^{d-1}(r)\cdot e^{-\frac{\tau}{2}}u_\rho\right) + e^{\frac{\tau}{2}}\sinh^{d-1}(r)\ r_\tau\ e^{-\frac{\tau}{2}} u_\rho\\
%&= \frac{1}{2}v + (d-1)\coth(r)\ r_\tau\ v +e^\frac{\tau}{2}\ \partial_\rho\left(\sinh^{d-1}(r)\cdot \partial_\rho\left[\frac{v}{e^\frac{\tau}{2}\sinh^{d-1}(r)}\right]\right) \\
%&\quad + \sinh^{d-1}(r)\ r_\tau\  \partial_\rho\left[\frac{v}{e^\frac{\tau}{2}\sinh^{d-1}(r)}\right]\\
%&= \frac{1}{2}v + (d-1)\coth(r)\ r_\tau\ v +\partial_\rho\left(\sinh^{d-1}(r)\cdot \partial_\rho\left[\frac{v}{\sinh^{d-1}(r)}\right]\right) \\
%&\quad +  e^{-\frac{\tau}{2}}\ r_\tau\ v_\rho  - (d-1) \coth(r)\ r_\tau\  v\\
%&= \frac{1}{2}v +\partial_\rho\left(v_\rho - (d-1)\coth(r)\ r_\rho v\right) +  e^{-\frac{\tau}{2}}\ r_\tau\ v_\rho\\
%&= \partial_\rho\left(v_\rho - (d-1)\coth(r)\ r_\rho v\right)+\frac{1}{2}v  +  \left((d-1)e^{\frac{\tau}{2}}+\frac{1}{2}\rho\right)\ v_\rho\\
%&= \partial_\rho\left(v_\rho + \left[(d-1)e^{\frac{\tau}{2}}+\frac{1}{2}\rho- (d-1)e^{\frac{\tau}{2}}\coth(r)\right]  v\right).
%\end{aligned}
%\]
%Summarizing, we have
\begin{equation}\label{eq:main_spherical_coordinates_symmetric_datum}
\begin{cases}
	v_\tau = \partial_\rho\left(v_\rho + \left[(d-1)e^{\frac{\tau}{2}}+\frac{1}{2}\rho- (d-1)e^{\frac{\tau}{2}}\coth(r)\right]  v\right),\quad &\rho\geq \rho_0(\tau),\ \tau>0,\\
	v(\tau, \rho_0(\tau))=0, &\tau>0,\\
	v(0,\cdot) = v_0, &\rho\geq \rho_0(\tau).
\end{cases}
\end{equation}
We can calculate a transient equilibrium $V$ of the previous equation by solving the ODE
\[
V_\rho + \left[(d-1)e^{\frac{\tau}{2}}+\frac{1}{2}\rho- (d-1)e^{\frac{\tau}{2}}\coth(r)\right]  V = 0
\]
leading, for an arbitrary function $C(\tau)$, to
\begin{equation}\label{eq:def_V}
V(\tau, \rho) = C(\tau)\cdot \sinh^{d-1}\big(r(\tau, \rho)\big)\cdot e^{- \frac{\left(\rho + 2(d-1)e^\frac{\tau}{2}\right)^2}{4}}.
\end{equation}
Then, recalling that $r=r(\tau,\rho)$ and $t=t(\tau)= e^\tau-1$, we get in the original coordinates the function
\[
V(t,r) = C(t)\cdot \sinh^{d-1}(r)\cdot e^{- \frac{\left(r + (d-1)(t+1)\right)^2}{4(t+1)}}.
\]
Note that this coincides with the equilibrium considered by Vázquez in \cite[Theorem 5.1]{vazquez1992asymptotic}.

Up to this point, $C(t)$ was arbitrary. Now we choose it in such a way that $V(t,r)$ has the same mass as $u(t,r)$ at each time, so we need to ask
\begin{equation}\label{eq_definC}
\mathcal{M}_{\mathcal{R}} = \int_{\rho_0(\tau)}^\infty V(\tau, \rho)\ \d \rho\quad\text{for all}\quad \tau>0.
\end{equation}
Note that although $C(t)$ defined in this way may not be constant. We can obtain the following estimates, whose proof is given in Appendix \ref{sec:UsefulBounds}.
	\begin{equation}\label{eq:bound_quotient_C(t)}
	-k_d\  e^{-m_d (t+1)}\leq \frac{\partial_t C(t)}{C(t)}\leq k_d\  e^{-m_d (t+1)},
\end{equation}
where
\[
m_d := \min \left(d-1, \frac{(d-1)^2}{16}\right),
\]
and
\begin{equation}\label{eq:limit_C(t)}
	\frac{2^{d-2}\mathcal{M}_{\mathcal{R}}}{\sqrt{\pi} }\leq C(t)\leq \frac{2^{d-2}\mathcal{M}_{\mathcal{R}}}{(1-e^{-(d-1)(t+1)})^{d-1}\left[\sqrt{\pi} - \frac{1}{\sqrt{\lambda_1 (t+1)}}e^{-\frac{\lambda_1 (t+1)}{4}}\right]}.
\end{equation}
Joining everything together we can write problem~\eqref{eq:main_spherical_coordinates_symmetric_datum} as
\[
%\begin{equation}\label{eq:main_spherical_coordinates_symmetric_datum_logarithmic_term}
	\begin{cases}
		v_\tau = \partial_\rho\left(v\cdot\partial_\rho\left[\ln\left(\frac{v}{V}\right)\right]\right),\quad &\rho\geq \rho_0(\tau),\ \tau>0,\\
		v(\tau, \rho_0(\tau))=0, &\tau>0,\\
		v(0,\cdot) = v_0, &\rho\geq \rho_0(\tau).
	\end{cases}
%\end{equation}
\]
with $V$ defined as in~\eqref{eq:def_V}. This leads to the definition of the entropy $H$ of the solution $v$ of the previous problem as
\[
%\begin{equation}\label{eq:def_entropy_radially_symmetryc}
H(\tau):= \int_{\rho_0(\tau)}^\infty v\cdot\ln\left(\frac{v}{V}\right)\ \d \rho.
%\end{equation}
\]
The next step is to obtain an estimate for the time derivative of $H(\tau)$. We can compute
\[
\partial_\tau H(\tau) = -\rho_0'(\tau)\left[v\ln\left(\frac{v}{V}\right)\right]_{\rho=\rho_0(\tau)} +\int_{\rho_0(\tau)}^\infty v_\tau - \int_{\rho_0(\tau)}^\infty v\cdot\frac{V_\tau}{V} + \int_{\rho_0(\tau)}^\infty v_\tau \cdot\ln\left(\frac{v}{V}\right).
\]
Let us consider each term by separate. First, define
\[
\gamma_0(\tau):=\lim\limits_{\rho\to\rho_0(\tau)} \frac{v}{V}= \lim\limits_{\rho\to\rho_0(\tau)}\left[ C(\tau)e^\frac{\tau}{2}u(t(\tau),r(\tau, \rho))e^{\frac{\left(\rho+2(d-1)e^\frac{\tau}{2}\right)^2}{4}}\right]= C(t)e^\frac{\tau}{2}u(t(\tau),0)e^{\lambda_1 e^\tau}.
\]
Clearly, since $u(t,0)$ is a positive finite quantity for every $t> 0$, we must have that $0<\gamma_0(\tau)<\infty$, and therefore
\[
\left[v\ln\left(\frac{v}{V}\right)\right]_{\rho=\rho_0(\tau)} = \left[\frac{v}{V}\cdot\ln\left(\frac{v}{V}\right)\cdot V\right]_{\rho=\rho_0(\tau)}=\gamma_0(\tau)\cdot \ln(\gamma_0(\tau))\cdot \underbrace{V(\tau, \rho_0(\tau))}_{=0}= 0.
\]

Next, the conservation of mass yields
\[
0=\partial_\tau \int_{\rho_0(\tau)}^\infty v\ \d\rho = \int_{\rho_0(\tau)}^\infty v_\tau\ \d\rho -\rho_0'(\tau)\underbrace{v(\tau, \rho_0(\tau))}_{=0} = \int_{\rho_0(\tau)}^\infty v_\tau \ \d\rho
\]
and therefore
\[
\int_{\rho_0(\tau)}^\infty v_\tau \ \d\rho =0.
\]
Finally, we can compute, using \eqref{eq:bound_quotient_C(t)},
\[
\frac{V_\tau}{V}=\frac{\partial_\tau C(t(\tau))}{C(t(\tau))} + \underbrace{(d-1)\left((d-1)e^\tau + \frac{1}{2}e^\frac{\tau}{2}\rho\right)(\coth(r(\tau,\rho))-1)}_{\geq 0}\geq -k_d e^{\tau-m_d e^\tau}.
\]
for all $\rho\geq \rho_0(\tau)$.

All of this together yields
\[
\partial_\tau H(\tau) \leq k_d\ e^{\tau-m_d e^\tau} \mathcal{M}_{\mathcal{R}}  - \int_{\rho_0(\tau)}^\infty v \cdot\left|\partial_\rho\left[\ln\left(\frac{v}{V}\right)\right]\right|^2,
\]
implying, by the log-Sobolev inequality that can be found in Corollary~\ref{cor:log_sobolev}, that
\[
\partial_\tau H(\tau) \leq k_d\ e^{\tau-m_d e^\tau}\mathcal{M}_{\mathcal{R}}  - H(\tau).
\]
Thus, we obtain, thanks to Gronwall's Lemma,
\[
e^{ \tau} H(\tau) - H(0)\leq k_d \mathcal{M}_{\mathcal{R}} \int_0^\tau  e^{ 2s-m_d e^s}  \ \d s.
\]
Let now $A(m_d)$ be such that
\[
e^{2s-m_d e^s}\leq e^{s-\frac{m_d}{2} e^s}\quad\text{for all}\quad s\geq A(m_d).
\]
Then,
\[
\begin{aligned}
e^{\tau} H(\tau) - H(0)&\leq k_d\mathcal{M}_{\mathcal{R}}\left( \int_0^{A(m_d)}  e^{2s-m_d e^s}  \ \d s + \int_{A(m_d)}^\tau  e^{s-\frac{m_d}{2} e^s}  \ \d s\right)\\[10pt]
& \leq k_d\mathcal{M}_{\mathcal{R}}\left( \tilde{A}(m_d) + \frac{2}{m_d}e^{-\frac{m_d}{2}\tau} \right)\leq K(d, \mathcal{M}_{\mathcal{R}})
\end{aligned}
\]
for some constant $K(d, \mathcal{M}_{\mathcal{R}})$. In total, we obtain
\[
H(\tau)\leq \left(K(d, \mathcal{M}_{\mathcal{R}}) + H(0)\right)e^{- \tau}.
\]
If we define
\[
H^*(v_0):= K(d, \mathcal{M}_{\mathcal{R}})+ H(0)
\]
then we get, thanks to the Csiszár-Kullback inequality,
\[
\|v-V\|_{L^1([\rho_0(\tau), \infty))}\leq \sqrt{H^*(v_0)}\cdot e^{-\frac{1}{2}\tau}.
\]
Writing it down in terms of $u$ (writing $H^*(u_0)$ instead of $H^*(v_0)$) this reads
\[
\int_{\rho_0(\tau)}^\infty \left|e^\frac{\tau}{2} u(t,r) - C(t)e^{-\frac{\left(r+(d-1)(t+1)\right)^2}{4(t+1)}}  \right|\ \sinh^{d-1}(r)\ \d \rho\leq \sqrt{H^*(u_0)}\cdot e^{-\frac{1}{2}\tau}.
\]
Changing variables $\d \rho = e^{-\tau/2}\d r, t= e^\tau-1$,
\[
\int_{0}^\infty \left| u(t,r) - C(t)\cdot\Gamma(t+1, r+(d-1)(t+1))  \right|\ \sinh^{d-1}(r)\ \d r\leq \sqrt{H^*(u_0)}\cdot (t+1)^{-\frac{1}{2}},
\]
where $\Gamma(t,r)$ stands for the Gaussian profile
\begin{equation}\label{eq:GaussianProfile}
	\Gamma(t,r):=\frac{1}{\sqrt{t}}e^{-\frac{r^2}{4t}}, \quad \forall r,t \in \mathbb{R}.
\end{equation}
Since the hyperbolic heat kernel $G_d(t,r)$ is a radial solution in on itself, the previous result still applies if we change $u$ by $G_d$, although paying perhaps with a small time translation so we make sure the initial entropy of the kernel is finite . With a small abuse of notation, we can also write $H(G_d)$ for the initial entropy of the kernel, after performing the change of variables $t\to \tau$ and $r\to \rho$, and we write $\mu (r) = \sinh (r)^{d-1} \d r$ for the radial part of the volume form of $\mathbb{H}^d$. We have proven the following.

\begin{thm}\label{thm:L^1_radially_symmetric_asymptotic}
	Let $u$ be a solution of~\eqref{eq:heatEquation_radial_solutions} such that $u_0(x)=u_0(r)$ for all $x\in\H^d$ and $\mathcal{M}_{\mathcal{R}}<\infty$. Then
	\[
	\| u(t,r) - C(t)\cdot\Gamma(t+1, r+(d-1)(t+1)) \|_{L^1(\R^+; \d\mu)}  \leq \sqrt{H^*(u_0)}\cdot (t+1)^{-\frac{1}{2}},
	\]
	and, in particular,
	\[
	\|\mathcal{M}_{\mathcal{R}} G_d(t,r) - C(t)\cdot\Gamma(t+1, r+(d-1)(t+1))  \|_{L^1(\R^+; \d\mu)}\leq\sqrt{H^*(G_d)} \cdot (t+1)^{-\frac{1}{2}},
	\]
	where $C(t)$ is defined by the relation given in \eqref{eq_definC}.
\end{thm}

Since
	\[
	1-(d-1)e^{-(d-1)t}\leq \left(1-e^{-(d-1)t}\right)^{d-1}
	\]
	we deduce from \eqref{eq:limit_C(t)} that
	\[
	\frac{2^{d-2}\mathcal{M}_{\mathcal{R}}}{\sqrt{\pi} }\leq C(t)\leq \frac{2^{d-2}\mathcal{M}_{\mathcal{R}}}{\left[1-(d-1)e^{-(d-1)(t+1)}\right]\cdot\left[\sqrt{\pi} - \frac{1}{\sqrt{\lambda_1 (t+1)}}e^{-\frac{\lambda_1 (t+1)}{4}}\right]},
	\]
	which implies
	\[
	\frac{2^{d-2}\mathcal{M}_{\mathcal{R}}}{\sqrt{\pi} }\leq C(t)\leq \frac{2^{d-2}\mathcal{M}_{\mathcal{R}}}{\sqrt{\pi}} + O(e^{-m_d(t+1)}),
	\]
	where we have $m_d := \min \left(d-1, \lambda_1 /4\right).$ This implies that in the previous result we are able to exchange $C(t)$ by its limit value $C_\infty = \frac{2^{d-2}\mathcal{M}_{\mathcal{R}}}{\sqrt{\pi}}$ without loosing speed of convergence, since
	\[
	\| u- C_\infty\Gamma \|_{L^1} \leq \| u - C(t)\cdot\Gamma \|_{L^1} + (C(t)-C_\infty)\|\Gamma \|_{L^1} \lesssim \sqrt{H^*(u_0)}\cdot (t+1)^{-\frac{1}{2}} + e^{-m_d(t+1)},
	\]
	and therefore there must exist a time $t_0$ depending on $m_d$ and $H^*(u_0)$ such that
	\[
	\| u- C_\infty\Gamma \|_{L^1} \lesssim  \sqrt{H^*(u_0)}\cdot (t+1)^{-\frac{1}{2}}\quad\text{for all}\quad t\geq t_0.
	\]
	This yields the following corollary.
	\begin{cor}\label{cor:L^1_radially_symmetric_asymptotic}
		Let $u$ be a solution of~\eqref{eq:heatEquation_radial_solutions} such that $u_0(x)=u_0(r)$ for all $x\in\H^d$ and $\mathcal{M}_{\mathcal{R}}<\infty$. Define $C_\infty = \frac{2^{d-2}\mathcal{M}_{\mathcal{R}}}{\sqrt{\pi}}$. Then there exists a $t_0$ such that, for all $t\geq t_0$,
		\[
		\| u(t,r) - C_\infty\ \Gamma(t+1, r+(d-1)(t+1)) \|_{L^1(\R^+; \d\mu)}  \lesssim \sqrt{H^*(u_0)}\cdot (t+1)^{-\frac{1}{2}},
		\]
		and, in particular,
		\[
		\|\mathcal{M}_{\mathcal{R}} G_d(t,r) - C_\infty\ \Gamma(t+1, r+(d-1)(t+1))  \|_{L^1(\R^+; \d\mu)}\lesssim\sqrt{H^*(G_d)} \cdot (t+1)^{-\frac{1}{2}}.
		\]
		
		As a consequence of this,
		\[
		\| u(t,x) - \mathcal{M}_{\mathcal{R}}P_d(t,x)  \|_{L^1(\mathbb{H}^d)}\lesssim \left(\sqrt{H^*(G_d)}+\sqrt{H^*(v_0)}\right) \cdot (t+1)^{-\frac{1}{2}}.
		\]
	\end{cor}
	
Since the difference of two solutions is itself a solution of the Heat Equation, then the $L^1-L^\infty$ regularizing effect described in \cite[Proposition 2.3]{vazquez1992asymptotic} applies to the function
\[
\tilde u(t,r):= u(t, r)- \mathcal{M}_{\mathcal{R}}G_d(t,r),
\]
implying that
\[
\|\tilde{u}(t,\cdot)\|_{L^\infty(\R^+)}\leq C(d)t^{-\frac{3}{2}}e^{-\lambda_1 t}\|\tilde{u}(t,\cdot)\|_{L^1(\mathbb{H}^d)},
\]
where $\lambda_1$ is the bottom of the spectrum of the Laplacian in $\mathbb{H}^d$. This $L^1-L^\infty$ regularizing effect and the previous theorem readily yield the following.

\begin{thm}\label{thm:radially_symmetric_L^infty asymtotic}
	Let $u$ a solution of~\eqref{eq:heatEquation_radial_solutions} such that $u_0(x)=u_0(r)$ for all $x\in\H^d$ and $\mathcal{M}_{\mathcal{R}}<\infty$. Then
	\[
	\| u(t,r) - \mathcal{M}_{\mathcal{R}}G_d(t,r)  \|_{L^\infty(\R^+)}\lesssim \left(\sqrt{H^*(G_d)}+\sqrt{H^*(v_0)}\right) \cdot e^{-\lambda_1 t}t^{-2}.
	\]
\end{thm}

\subsection{Horospheric solutions}\label{sect:horospheric}

Now let $u$ be a classical solution to \eqref{eq:heatEquationHyperbolic} with initial data $u_0$ depending only on the distance to a certain horosphere $\mathcal{H}$. Then, following~\eqref{laplacianParabolic}, $u$ only depends on the distance to $\mathcal{H}$ as well and satisfies 
\[
u(t,r,y)=u(t,r), \quad \text{for any } y \in \mathbb{R}^{d-1},\quad \Delta_{\H^d} = \frac{1}{e^{(d-1)r}} \partial_r \left(e^{(d-1)r}\cdot \partial_r u\right).
\]
Equation~\eqref{eq:heatEquationHyperbolic} then becomes
\begin{equation}\label{eq:heatEquation_horospheric_solutions}
	\begin{cases}
		u_t = e^{-(d-1)r}\partial_r\left(e^{(d-1)r}\cdot \partial_r u\right), \quad &r\in\R, t>0,\\
		u(0,r)=u_0(r), & r\in\R.
	\end{cases}
\end{equation}
We say that $u$ is a \textit{horospherical solution} to the heat equation. In opposition to radially symmetric solutions, horospherical solutions to \eqref{eq:heatEquationHyperbolic} are not in $L^{1} (\mathbb{H}^{d})$ in general. However, the following notion of mass make sense for this type of solutions:
\[
\|u(t,\cdot)\|_{L^1(\mathbb{R}, e^{(d-1)r} \d r)} = \int_{\R} u(t,r)\ e^{(d-1)r}\ \d r,
\]
and it can be checked that this quantity does not depend on $t$ if $u$ solves the heat equation (i.e., this mass is preserved). From now on, we write
\[
\mathcal{M}_{\mathcal{H}} := \|u\|_{L^1(\mathbb{R}, e^{(d-1)r} \d r)} = \int_{\R} u(t,r)\ e^{(d-1)r}\ \d r = \int_{\R} u_0 (r)\ e^{(d-1)r}\ \d r,
\]
and we refer to this quantity as the \textit{horospherical mass} of $u$. Furthermore, we will assume the integrability condition $\mathcal{M}_{\mathcal{H}}<+\infty$.

Now define 
\begin{equation}\label{definv}
v(\tau, \rho):=e^\frac{\tau}{2}\cdot e^{(d-1)r (\tau, \rho)} \cdot u\big(t(\tau), r(\tau,\rho)\big), \quad \tau \geq 0, \quad \rho \in \mathbb{R},
\end{equation}
where, as in the previous sections, we considered
\[
t(\tau)=e^\tau-1,\quad r(\tau,\rho)= (d-1)e^\tau  + e^\frac{\tau}{2}\rho.
\]
Then, we expect $v$ to satisfy a Fokker-Planck type equation. We check that this is in fact the case, and after some easy computations we get that $v$ solves
%\[
%v_\tau = \frac{v}{2}+ (d-1)r_\tau v + e^{\frac{\tau}{2}}e^{(d-1)r}t_\tau u_t +e^{\frac{\tau}{2}}e^{(d-1)r} r_\tau u_r
%\]
%and 
%\begin{equation}\label{firstDerivativeV}
%	v_\rho = (d-1)r_\rho v + e^{\frac{\tau}{2}}e^{(d-1)r}r_\rho u_r,
%\end{equation}
%thus, using \eqref{derivatives}, we get after rearranging the terms that $v$ solves
\begin{equation}\label{fokkerPlanckParabolic}
	\begin{cases}
	v_\tau = \partial_\rho \left( v_\rho + \frac{\rho}{2} v \right), \quad &\rho \in \mathbb{R},\ \tau>0,\\
		v(0,\cdot) = v_0, &\rho\geq \rho_0(\tau),
	\end{cases}
\end{equation}
where we define $v_0 (\rho) := e^{(d-1)r(0,\rho)}u_0 (r(0, \rho))$.

 The equilibrium for \eqref{fokkerPlanckParabolic} is obtained by solving the equation $ v_\rho + \frac{\rho}{2} v =0$. It is straightforward that the solution to this equation is given by
\[
V(\tau, \rho) = C(\tau) \cdot e^{(d-1)r(\tau,\rho)} e^{-\frac{\left(\rho+2(d-1)e^{\frac{\tau}{2}}\right)^2}{4}} = C(\tau)e^{-\frac{\rho^2}{4}} , \quad \forall \tau \geq 0, \quad \rho \in \mathbb{R},
\]
and we can write this function in terms of the variables $t,r$ as
\begin{equation}\label{asymptoticProfileParabolicTauRho}
V(t,r) = C(t) \cdot  e^{(d-1)r} \cdot e^{-\frac{(r+(d-1)(t+1))^2}{4(t+1)}}= C(t)e^{-\frac{(r-(d-1)(t+1))^2}{4(t+1)}}, \quad \forall t \geq 0, \quad r \in \mathbb{R}.
\end{equation}
Since $C(\tau)$ is the variable needed to make sure that $\|v(\tau,\cdot)\|_{L^1}= \|V(\tau,\cdot)\|_{L^1}$, it can be easily verified that in this case
\[
C(\tau)= \mathcal{M}_{\mathcal{H}},\qquad V(\tau, \rho)= \mathcal{M}_{\mathcal{H}}\  e^{-\frac{\rho^2}{4}}.
\]
Note also that equation~\eqref{fokkerPlanckParabolic} can be rewritten now as
\[
\begin{cases}
	v_\tau = \partial_\rho \left( v\cdot\partial_\rho \ln\left(\frac{v}{V}\right) \right), \quad &\rho \in \mathbb{R},\ \tau>0,\\
	v(0,\cdot) = v_0, &\rho\geq \rho_0(\tau),
\end{cases}
\]
This formulation suggest studying the large-time behaviour of $v$ by means of the convergence of the relative entropy
\[
H(\tau) := \int_{\R} v \cdot \ln \left(\frac{v}{V}\right) \, \d \rho, \quad \forall \tau \geq 0, 
\] 
where $v$ and $V$ are given by \eqref{definv} and \eqref{asymptoticProfileParabolicTauRho} respectively. By computing the first derivative of $H$, we get
\[
\begin{split}
	\partial_\tau H (\tau) & = \int_{\R} v_\tau \, \d \rho + \int_{\R} v_\tau \cdot\ln\left(\frac{v}{V}\right) \, \d \rho \\
%	& = \left[ v \partial_\rho \left( \ln \left(\frac{v}{V}\right) \right) \ln \left(\frac{v}{V}\right) \right]_{-\infty}^{+\infty} - \int_{-\infty}^{+\infty} v \abs{\partial_\rho \ln \left(\frac{v}{V}\right)}^2 \, \d \rho \\
	& =  - \int_{\R} v \abs{\partial_\rho \ln \left(\frac{v}{V}\right)}^2 \, \d \rho,
\end{split}
\]
where for the second identity we have used the conservation of mass.

From here, after applying a log-Sobolev inequality from Lemma~\ref{lem:curvature-logsob} with constant $\lambda=1$, Gronwall's Lemma and the Csiszár-Kullback inequality, and then undoing the change of variables, we prove the main result of this section. In order to keep everything on the original variables, we will write $H(u_0)$ instead of $H(v_0)=H(0)$, since the meaning of this symbol is clear. Also, we recall the notation $\d \nu (r) = e^{(d-1)r} \d r$.

\begin{thm}\label{thm:convregece_horoshepric_L^1}
	Let $u$ be a solution of~\eqref{eq:heatEquation_horospheric_solutions} such that $u_0(x)=u_0(r)$ for all $x\in\H^d$, and $\mathcal{M}_{\mathcal{H}}<\infty$. Then,
	\[
	\|u(t,r)- \mathcal{M}_{\mathcal{H}}\ \Gamma(t, r+(d-1)t) \|_{L^1(\R, \d\nu)} \lesssim \sqrt{H(u_0)} \cdot t^{-\frac{1}{2}}
	\]
%	\[
%	\displaystyle\int_{\R} \abs{u(t,r)- \mathcal{M}_{\mathcal{H}}\ \Gamma(t+1, r+(d-1)(t+1))} e^{(d-1)r} \, \d r \leq \sqrt{H(u_0)} \cdot (t+1)^{-\frac{1}{2}}
%	\]
	where $\Gamma(t,r)$ stands for the standard Gaussian, defined in \eqref{eq:GaussianProfile}.
\end{thm}

Indeed, it can be verified that the fundamental solution of~\eqref{eq:heatEquation_horospheric_solutions} is in fact $\Gamma(t, r+(d-1)t)$, and any integrable solution of~\eqref{eq:heatEquation_horospheric_solutions} can be written as
\[
u(t, r)=\int_\R \Gamma(t, r-R+(d-1)t)\ u_0(R)\ \d R.
\]
As a consequence, we can deduce an smoothing effect by means of
\[
\begin{aligned}
	u(2t, r)&- \mathcal{M}_{\mathcal{H}}\ \Gamma(2t+1, r+(d-1)(2t+1))\\[10pt]
	&= \int_\R \Gamma(t, r-R+(d-1)t)\cdot \left(u(t, R)- \mathcal{M}_{\mathcal{H}}\ \Gamma(t+1, R+(d-1)(t+1))\right)\ \d R
\end{aligned}
\]
which implies
\[
\begin{aligned}
	&|u(2t, r)- \mathcal{M}_{\mathcal{H}}\ \Gamma(2t+1, r+(d-1)(2t+1))|\\[10pt]
	&= \int_\R \frac{\Gamma(t, r-R+(d-1)t)}{e^{(d-1)R}}\cdot |u(t, R)- \mathcal{M}_{\mathcal{H}}\ \Gamma(t+1, R+(d-1)(t+1))| e^{(d-1)R}\ \d R.
\end{aligned}
\]
Using now that
\[
\frac{\Gamma(t, r-R+(d-1)t)}{e^{(d-1)R}}\lesssim t^{-\frac{1}{2}}
\]
and Theorem~\ref{thm:convregece_horoshepric_L^1} we obtain the following.

\begin{thm}\label{thm:convregece_horoshepric_L^infty}
	Let $u$ be a solution of~\eqref{eq:heatEquation_horospheric_solutions} such that $u_0(x)=u_0(r)$ for all $x\in\H^d$, and $\mathcal{M}_{\mathcal{H}}<\infty$. Then,
	\[
	\|u(t,r)- \mathcal{M}_{\mathcal{H}}\ \Gamma(t+1, r+(d-1)(t+1)) \|_{L^\infty(\R)} \lesssim \sqrt{H(u_0)} \cdot (t+1)^{-1}.
	\]
\end{thm}
%
%
%Then, using the relation given in \eqref{definv}, we expect the asymptotic profile of a horospherical solution to \eqref{eq:heatEquationHyperbolic} to be given by
%\begin{equation}\label{asymptoticProfileParabolic}
%	u_{\infty}(t,r)= \frac{C(t)}{\sqrt{t+1}} \cdot e^{-\frac{(r+(d-1)(t+1))^2}{4(t+1)}}, \quad \forall t \geq 0, \quad r \in \mathbb{R}.
%\end{equation}
%Since $u_{\infty}$ is an horospherical solution to \eqref{eq:heatEquationHyperbolic}, we can compute the precise value of $C(t)$ using the conservation of the horospherical mass. In fact, note that
%\[
%\int_{-\infty}^{\infty}u_{\infty} (0,r) e^{(d-1)r} \, \d r = C(0) \int_{-\infty}^{+\infty} e^{-\frac{(r-(d-1)^2)}{4}} \, \d r = 2 C(0) \sqrt{\pi},
%\]
%so we conclude that, for any $t \geq 0$,
%\[
%2 C(0) \sqrt{\pi} = \int_{-\infty}^{\infty}u_{\infty} (0,r) e^{(d-1)r} \, \d r = \frac{C(t)}{\sqrt{t+1}} \int_{-\infty}^{+\infty} e^{-\frac{(r-(d-1)(t+1))^2}{4(t+1)}} \, \d r = 2C(t) \sqrt{\pi},
%\]
%yielding $C(t) = C(0)$ for some $C(0)>0$.
% 
%{\color{blue} After considering the drifted time $\tilde{t}= t+1$ and choosing $C(0) = 1/(2\sqrt{\pi})$ we have that $u_{\infty}$ coincides with a drifted one dimensional Gaussian, which is precisely the asymptotic profile computed in [VAZQUEZ]}

\section{Asymptotic convergence of non-symmetric solutions.}\label{sect:convergence_integrable_no_rates}

Let us study now the problem
\begin{equation}\label{eq:laplace_beltrami_polar_coordinates_general data}
	\begin{cases}
		\partial_t u = \frac{1}{\sinh^{d-1}(r)}\partial_r\left(\sinh^{d-1}(r)\cdot \partial_r u\right) + \frac{1}{\sinh^2(r)}\Delta_{\mathbb{S}^{d-1}}u,\quad &x\in\H^d,\ t>0, \\[10pt]
		u(0,x)=u_0(x)\geq 0,&x\in\H^d,
	\end{cases}
\end{equation}
for a certain non-negative initial datum $u_0\in L^1(\H^d)$ that does not necessarily present the radial symmetry required in Section~\ref{sect:elliptic}. In this brief section we will show the main ideas on how to derive the result~\eqref{eq:convergence_papageorgiou} in the case $p=1$. Let us write in what follows $y=(r_y, \theta_y)$ and $x=(r_x, \theta_x)$ in order to better distinguish between integration variables, fixed and variable points, etc...

Let \( y \in \mathbb{H}^d \setminus \{\mathcal{O}\} \), and, for simplicity of notation, write $P_d$ for the fundamental solution centred at the origin $\mathcal{O}$. We begin by studying the long-time behaviour of the difference between the heat kernels \( P_d \) and \( P_d^y \). It was shown in \cite{vazquez1992asymptotic} for the three-dimensional case, and later extended in \cite{APZ} for general dimension \( d \), that
\[
\lim_{t \to \infty} \| P_d - P_d^y \|_{L^1(\mathbb{H}^d)} > 0.
\]
Here, one can improve on this result by incorporating geometric information. Let $x = (r_x, \theta_x)\in\H^d$ and
denote by $\langle\cdot, \cdot \rangle$ the Euclidean scalar product in $\R^d$. Let also $l\in(0,1)$. For any $\ell \in (-l  \sqrt{t}, l \sqrt{t})$, define $r_{\ell} (t) = (d-1)t+ \ell$ for all $t \in (0,+\infty)$. Then,  we have that $p_t=p_t(r_\ell(t), \theta_x)$ denotes a unique point in $\H^d$. Define finally
\begin{equation}\label{eq:limitQuotient}
	\varphi (y, \theta_x):= \displaystyle\lim_{t \to +\infty} \frac{P_d^y (t,p_t)}{ P_d (t,p_t)} .
\end{equation}
It follows from \cite[Proposition 3.7]{APZ} that $\varphi$ is a well defined continuous function on $\H^d \times \S^{d-1}$, and~\cite[Appendix]{Papageorgiou2025} provides the specific formula
\[
\varphi(y, \theta_x)=\varphi(r_y, \theta_y, \theta_x)=[\cosh(r_y)-\sinh(r_y)\cdot\left\langle \theta, \theta_y\right\rangle]^{-(d-1)}.
\]
Since this article is oriented to readers with a background on PDEs, in Appendix~\ref{sect:ExplicitVarPhi} we offer a derivation of $\varphi$ in the cases $d \in \{2,3\}$ that employs only some analysis and the explicit formula of the kernels that can be found in~\cite{Davies_1989,Grigor'yan,grigoryan1998heat}. Our approach holds in every dimension $d\geq 2$, but the computations become exponentially complicated, so it is more sensible to utilize directly the results from~\cite[Appendix]{Papageorgiou2025}.

From the formula of $\varphi$, we get the following result. It's proof is left to the reader as an exercise, taking into account~\cite[Proposition A.1]{LM}. The interested reader is invited to check also Proposition \ref{propVarPhi}.

\begin{prp}\label{prp:convergence_kernels_varphi}
	For any $y \in \mathbb{H}^d$, we have that 
	\[
	%\begin{equation} \label{differenceHeatKernels}
		\lim_{t \to \infty} \left\| P_d^y(t,\cdot) - \varphi(y, \cdot) \, P_d(t,\cdot) \right\|_{L^1(\mathbb{H}^d)} = 0.
	%\end{equation}
	\]
\end{prp}
This result suggests that the asymptotic profile of a general solution $u$ is going to be given by the product of the fundamental solution centred at the origin, $P_d(t,x)=G_d(t,r_x)$, and the weighted directional mass
\begin{equation}\label{eq:def_Phi}
\Phi(\theta_x):= \int_{\H^d} u_0(y)\, \varphi(y,\theta_x)\ \d\mu(y).
\end{equation}
In other words, we would like to prove that
\[
\lim\limits_{t\to\infty}\|u(t,x)-\Phi(\theta_x)G_d(t,r_x)\|_{L^1(\H^d)}=0.
\]

However, note that the function $\Phi(\theta_x)$ may not even be defined, since $\varphi(y,\theta_x)$ is not uniformly bounded. Indeed, the optimal bound for $\varphi$ that can be attained is
\begin{equation}\label{eq:bound_varphi}
\varphi(y,\theta_x)\leq \left[ \cosh (r_y)-\sinh (r_y)  \right]^{-(d-1) }= e^{(d-1) r_y}
\end{equation}
and this bound is attained whenever $\theta_x=\theta_y$. Therefore, we need the integral in~\eqref{eq:def_Phi} to be well defined and the function $\Phi(\theta_x)G_d(t,x)$ to be integrable. Since
\[
\begin{aligned}
\|\Phi(\theta_x)G_d(t,x)\|_{L^1(\H^d)}& = \left(\int_{\mathbb{S}^{d-1}}\Phi(\theta_x)\ \d \theta_x\right) \cdot \int_0^\infty G_d(t,r_x)\, \sinh^{d-1}(r_x)\ \d r_x\\[10pt] &=\frac{1}{|\mathbb{S}^{d-1}|} \int_{\mathbb{S}^{d-1}}\Phi(\theta_x)\ \d \theta_x
\end{aligned}
\]
we need to assume that $\Phi\in L^1(\mathbb{S}^{d-1})$. Having this one obtains the following theorem. Its proof is left to the reader as an exercise. It also corresponds to~\eqref{eq:convergence_papageorgiou} in the case $p=1$.

\begin{thm}\label{thm:convergence_without_rate}
	Let $u$ be the solution to problem~\eqref{eq:laplace_beltrami_polar_coordinates_general data} for a non-negative initial datum $u_0\in L^1(\H^d)$ such that $\Phi \in L^1(\mathbb{S}^{d-1})$, defined as in~\eqref{eq:def_Phi}. Then
	\[
	\lim\limits_{t\to\infty}\|u(t,x)-\Phi(\theta_x)P_d (t,x)\|_{L^1(\H^d)}=0.
	\]
\end{thm}

\begin{rem}\label{rem:varphi_continuous_phi_bounded}
	Since $\varphi$ is a continuous function, a sufficient condition for $\Phi \in L^\infty(\mathbb{S}^{d-1})\subset L^1(\mathbb{S}^{d-1})$ is that $u_0$ has compact support.
\end{rem}

Following~\eqref{eq:bound_varphi}, a sufficient condition  for $\Phi \in L^1(\mathbb{S}^{d-1})$ is that
\begin{equation}\label{eq:condition_exponential_initial_datum}
\int_{\H^d} u_0(y)\, e^{(d-1) r_y}\ \d\mu(y) = \int_{\mathbb{S}^{d-1}}\int_0^\infty u_0(r_y, \theta_y)\, e^{(d-1) r_y}\, \sinh^{d-1}(r_y)\ \d r_y\ \d \theta_y\leq C<\infty,
\end{equation}
for a certain positive constant $C=C(u_0)$. Notice how this condition matches the one in~\eqref{eq:condition_papgeorgiiou} for $\Phi_1$. With this at hand, the following lemma unfolds.

\begin{lem}\label{lem:boundedness_of_Phi}
Let $u$ be the solution to problem~\eqref{eq:laplace_beltrami_polar_coordinates_general data} for a non-negative initial datum $u_0\in L^1(\H^d)$ such that condition~\eqref{eq:condition_exponential_initial_datum} is satisfied. Then $\Phi(\theta_x)$ is uniformly bounded.
\end{lem}
\section{Convergence with rates of non-symmetric solutions}\label{sect:convergence_integrable_with_rates}
In this section we improve on the results in~\cite{Papageorgiou2025} by studying the asymptotic behaviour of solutions to the heat equation with general compactly supported initial data using entropy methods, in the spirit of Section \ref{sect:convergence_symmetries}.
%Later, in Appendix \ref{sec:L1Data}, we generalize the results obtained here to the case \( u_0 \in L^1(\mathbb{H}^d) \) with suitable decay, but for simplicity we restrict ourselves to the compactly supported case in this section.

\subsection{Entropy for general solutions}

Let us define
\[
v(\tau, \rho, \theta):=e^\frac{\tau}{2}\cdot\sinh^{d-1}\big(r(\tau,\rho)\big)\cdot u\big(t(\tau), r(\tau,\rho), \theta\big),
\]
the function $V(\tau, \rho)$ as in~\eqref{eq:def_V} but this time choosing $C(\tau)$ such that $\int V=1$, and the function
\[
\mathcal{N}(\tau, \theta) := \int_{\rho_0(\tau)}^\infty v(\tau, \rho, \theta) \d \rho.
\]
Then, following~\eqref{eq:laplace_beltrami_polar_coordinates}, the function  $v$ satisfies
\begin{equation}\label{eq:main_spherical_coordinates_general_datum}\notag
	\begin{cases}
		v_\tau = \partial_\rho\left(v\cdot\partial_\rho\left[\ln\left(\frac{v}{\mathcal{N}\cdot V}\right)\right]\right) +\frac{e^\tau}{\sinh^{2}(r)}\Delta_{\mathbb{S}^{d-1}}v ,\quad &\rho\geq \rho_0(\tau),\ \tau>0,\\
		v(\tau, \rho_0(\tau))=0, &\tau>0,\\
		v(0,\cdot) = v_0, &\rho\geq \rho_0(\tau).
	\end{cases}
\end{equation}
Notice how this equation has a spherical term involving $\Delta_{\mathbb{S}^{d-1}}v$, hence the need to introduce the \textit{directional mass} function $\mathcal{N}$. On the other hand, since this term is multiplied by a factor $\frac{e^\tau}{\sinh^{2}(r)}$, we expect it to be small enough at regions of the form $r\sim (d-1)t$ (where most of the mass is located) so that $L^1$ convergence retains the same speed $t^{-1/2}$.

We define then the entropy
\begin{equation}\label{eq:def_entropy_general_data}\notag
	H(\tau):= \int_{\mathbb{S}^{d-1}}\int_{\rho_0(\tau)}^\infty v\cdot\ln\left(\frac{v}{\mathcal{N}\cdot V}\right)\ \d \rho\ \d \theta,
\end{equation}
and we can compute, as in the previous sections,
\[
\begin{aligned}
	\partial_\tau H(\tau) & =  - \int_{\mathbb{S}^{d-1}}\int_{\rho_0(\tau)}^\infty v\cdot\frac{V_\tau}{V} - \int_{\mathbb{S}^{d-1}}\int_{\rho_0(\tau)}^\infty v\cdot\frac{\mathcal{N}_\tau}{\mathcal{N}}  + \int_{\mathbb{S}^{d-1}}\int_{\rho_0(\tau)}^\infty v_\tau \cdot\ln\left(\frac{v}{\mathcal{N}\cdot V}\right)\\[10pt]
	&=  - \int_{\mathbb{S}^{d-1}}\int_{\rho_0(\tau)}^\infty v\cdot\frac{V_\tau}{V} - \int_{\mathbb{S}^{d-1}}\mathcal{N}_\tau - \int_{\mathbb{S}^{d-1}}\int_{\rho_0(\tau)}^\infty v\left|\partial_\rho \cdot\ln\left(\frac{v}{\mathcal{N}\cdot V}\right)\right|^2\\[10pt]
	& \ + \int_{\mathbb{S}^{d-1}}\int_{\rho_0(\tau)}^\infty \frac{e^\tau}{\sinh^{2}(r)}\Delta_{\mathbb{S}^{d-1}}v \cdot\ln\left(\frac{v}{\mathcal{N}\cdot V}\right).
\end{aligned}
\]
Let us estimate some of the terms in the right-hand side. We can use that, after some computations,
\[
\mathcal{N}_\tau = \int_{\rho_0(\tau)}^\infty \frac{e^\tau}{\sinh^{2}(r)} \Delta_{\mathbb{S}^{d-1}}v(\tau, \rho, \theta) \d \rho,
\]
which implies that
\[
\int_{\mathbb{S}^{d-1}}\mathcal{N}_\tau=0.
\]
On the other hand, integrating by parts and using the periodicity of $v$ on the directional variable $\theta$, we get to
\[
\begin{aligned}
	\int_{\mathbb{S}^{d-1}}\int_{\rho_0(\tau)}^\infty& \frac{e^\tau}{\sinh^{2}(r)}\Delta_{\mathbb{S}^{d-1}}v \cdot\ln\left(\frac{v}{\mathcal{N}\cdot V}\right) = \int_{\rho_0(\tau)}^\infty \frac{e^\tau}{\sinh^{2}(r)}\int_{\mathbb{S}^{d-1}}\Delta_{\mathbb{S}^{d-1}}v \cdot (\ln\left(v\right) - \ln\left(\mathcal{N}\right))\\[10pt]
	& = \int_{\rho_0(\tau)}^\infty \frac{e^\tau}{\sinh^{2}(r)}\int_{\mathbb{S}^{d-1}}v \cdot \left(\frac{\Delta_{\mathbb{S}^{d-1}} v}{v} - \frac{|\nabla_{\mathbb{S}^{d-1}}v|^2}{v^2}\right) -\Delta_{\mathbb{S}^{d-1}}v\, \ln(\mathcal{N}) \\[10pt]
	& \leq -\int_{\rho_0(\tau)}^\infty \frac{e^\tau}{\sinh^{2}(r)}\int_{\mathbb{S}^{d-1}}\Delta_{\mathbb{S}^{d-1}}v\, \ln(\mathcal{N}) = -\int_{\mathbb{S}^{d-1}}\mathcal{N}_\tau\, \ln(\mathcal{N})  
\end{aligned}
\]

Now, this directional mass distribution $\mathcal{N}$ makes it so the equality
\[
\int_{\rho_0(\tau)}^\infty v(\tau, \rho, \theta) \d \rho = \int_{\rho_0(\tau)}^\infty \mathcal{N}(\tau, \theta) V(\tau,\rho) \d \rho
\]
holds, in order to apply log-Sobolev inequalities. Putting all this together yields
\[
\begin{aligned}
	\partial_\tau H(\tau) & \leq   - \int_{\mathbb{S}^{d-1}}\int_{\rho_0(\tau)}^\infty v\cdot\frac{V_\tau}{V} -\int_{\mathbb{S}^{d-1}}\mathcal{N}_\tau\, \ln(\mathcal{N})   - \lambda H(\tau).
\end{aligned}
\]

Similarly as in Section~\ref{sect:elliptic}, we have
\[
\int_{\mathbb{S}^{d-1}}\int_{\rho_0(\tau)}^\infty v\cdot\frac{V_\tau}{V}\geq - k_d\ e^{\tau-m_d e^\tau} \int_{\mathbb{S}^{d-1}}\int_{\rho_0(\tau)}^\infty v = - k_d\ \|u_0\|_{L^1(\H^d)}\ e^{\tau-m_d e^\tau},
\]
and therefore,
\begin{equation}\label{eq:entropy_with_N}
	\begin{aligned}
		\partial_\tau H(\tau) & \leq  k_d\ \|u_0\|_{L^1(\H^d)}\ e^{\tau-m_d e^\tau} -\int_{\mathbb{S}^{d-1}}\mathcal{N}_\tau\, \ln(\mathcal{N})   - \lambda H(\tau),
	\end{aligned}
\end{equation}
so in order to conclude the convergence of the relative entropy we need a good estimate on $\mathcal{N}$ and $\partial_\tau \mathcal{N}$, which we will derive in the next section.

\subsection{The directional mass distribution}\label{sect:directional_mass_distribution}

This subsection is devoted to study the behaviour of the function $\mathcal{N}(\tau,\theta)$.

\begin{prp}\label{lem:bounds_N}
	There exists a positive constant $C_1(u_0, d)$ depending only on the initial datum and the dimension $d$ such that
	\[
	0<C_1(u_0, d)\leq \mathcal{N}.
	\]
\end{prp}
\begin{proof}
	A comparison argument by below with a radially symmetric initial datum with compact support and Theorem~\ref{thm:L^1_radially_symmetric_asymptotic} yield the bound for $\mathcal{N}$.
\end{proof}

%Next, we estimate $\partial_\tau \mathcal{N}$.
%In order to do so we will need two hypothesis. Let us assume that there exist a couple of constants $K_1, K_2$ depending only on the dimension and on $u_0$ such that
%\begin{equation}\label{eq:condition_convergnece_N}
%	\begin{aligned}
	%	&\int_0^\infty \int_{\mathbb{S}^{d-1}}\int_{0}^\infty u_0(r_y, \theta_y)\, t^{-\frac{3}{2}}e^{-\frac{\lambda^2}{4t} - \frac{(d+2)\lambda}{2}}\, \sinh^{d+1}(r_y) \, \sinh^{d-1}(r_x)\,\d r_y\,\d \theta_y  \, \d r_x\leq K_1(d, u_0),\\[10pt]
	%	&\int_0^\infty \int_{\mathbb{S}^{d-1}}\int_{0}^\infty u_0(r_y, \theta_y)\, t^{-\frac{3}{2}}e^{-\frac{\lambda^2}{4t} - \frac{d\lambda}{2}}\, \sinh^{d}(r_y) \, \sinh^{d-2}(r_x)\,\d r_y\,\d \theta_y  \, \d r_x\leq K_2(d, u_0).
	%	\end{aligned}
%\end{equation}
%
%\ag{IT IS NOT ENTIRELY CLEAR WHETHER THIS IS OPTIMAL FOR OUR PURPOSES OR NOT, see the proof below. ALSO, I THINK THAT $u_0$ 	with compact support is a sufficient condition for the previous hypothesis to be true, but we would have to prove it}

Next, we estimate $\partial_\tau \mathcal{N}$. It is more convenient to study it in terms of the original variables, except for the time variable for which we simply write $t=t(\tau)\sim e^\tau$ when needed. At a point $\theta_x\in\mathbb{S}^{d-1}$,
\begin{equation}\label{eq:derivative_N_in_time}
	\partial_\tau \mathcal{N}(\tau,\theta_x) = e^\tau \int_0^\infty \sinh^{d-3}(r_x) \Delta_{\mathbb{S}^{d-1}, \theta_x} u(t,r_x, \theta_x)\ \d r_x.
\end{equation}
The first step is to study then $\Delta_{\mathbb{S}^{d-1}, \theta_x} u$.	Since the solution $u$ satisfies $u=\int u_0\, G_d^y\, \d \mu (y)$, then
\begin{equation}\label{eq:spherical_laplacian_of_u_as_a_convolution}
	\Delta_{\mathbb{S}^{d-1}, \theta_x} u(t,r_x, \theta_x) = \int_{\mathbb{S}^{d-1}}\int_{0}^\infty u_0(r_y, \theta_y)\, [\Delta_{\mathbb{S}^{d-1}, \theta_x} G_d(t,\lambda(x,y))]\, \sinh^{d-1}(r_y)\ \d r_y\ \d \theta_y
\end{equation}
where $\lambda(x,y)$ is defined as in~\eqref{eq:formula_lambda}. In order to compute $\Delta_{\mathbb{S}^{d-1}, \theta_x} G_d$, we fix $y \in \mathbb{H}^d$ and thus we define 
\[
s (\theta) := \textup{dist}_{\mathbb{S}^{d-1}} (\theta, \theta_y), \quad \forall \theta \in \mathbb{S}^{d-1}.
\]
Then, since we have that $\lambda = \lambda (r_x, s(\theta_x))$ when $y$ is fixed, we have that 
\[
\Delta_{\mathbb{S}^{d-1}, \theta_x} G_d (t, \lambda) = \partial_s G_d \Delta_{\mathbb{S}^{d-1}} s + \partial_{ss} G_d | \nabla_{\mathbb{S}^{d-1}} s|^2,
\]
and since
\[
\Delta_{\mathbb{S}^{d-1}} s = (d-1) \cot (x) \quad \textup{and} \quad | \nabla_{\mathbb{S}^{d-1}} s|^2 =1,
\]
we conclude that
\[
\Delta_{\mathbb{S}^{d-1}, \theta_x} G_d= (d-1) \cot (s) \partial_s G_d  + \partial_{ss} G_d.
\]
See~\cite{petersen2016riemannian} for a proof of these formulas. Now, taking into account that 
\[
\partial_s G_d = \partial_\lambda G_d \partial_s \lambda , \quad \partial_{ss} G_d = \partial_{\lambda \lambda} G_d (\partial_s \lambda)^2+\partial_\lambda G_d \partial_{ss} \lambda,
\]
and the recurrence formula \eqref{eq:recurrence_derivative_kernels}, we get after simplifying that
\[
\Delta_{\mathbb{S}^{d-1}, \theta_x} G_d = 4 \pi^2 e^{2(d+1)t}\sinh (r_x)^2 \sinh (r_y)^2 \sin (s)^2 G_{d+4}-2 \pi d e^{dt}\sinh (r_x) \sinh (r_y) \cos (s) G_{d+2}.
\]
Substituting this equation into~\eqref{eq:derivative_N_in_time} and ~\eqref{eq:spherical_laplacian_of_u_as_a_convolution}  we readily get, with $t=t(\tau)\sim e^\tau$,
\begin{equation}\label{eq:formula_for_derivative_N_in_time}
	\partial_\tau \mathcal{N}(\tau,\theta_x) =  4\pi^2 t e^{2(d+1)t}\cdot \mathcal{I}-  2\pi d t e^{dt}\cdot\mathcal{II},
\end{equation}
with
\[
\begin{aligned}
	&\mathcal{I} := \int_0^\infty \int_{\mathbb{S}^{d-1}}\int_{0}^\infty u_0(r_y, \theta_y)\, \sinh^{d+1}(r_y)\,\sin (\vartheta)^2\, G_{d+4}(t,\lambda(x,y))\, \sinh^{d-1}(r_x)\,\d r_y\,\d \theta_y  \, \d r_x,\\[10pt]
	&\mathcal{II} :=\int_0^\infty \int_{\mathbb{S}^{d-1}}\int_{0}^\infty u_0(r_y, \theta_y)\, \sinh^{d}(r_y)\, \cos (\vartheta) \, G_{d+2}(t,\lambda(x,y))\, \sinh^{d-2}(r_x)\,\d r_y\,\d \theta_y  \, \d r_x.
\end{aligned}
\]	
We recall that we are writing, for simplicity, $\vartheta$ as the angle formed by $\theta_x,\theta_y$, i.e.,
\[
\vartheta(\theta_x,\theta_y) = \arccos(\langle\theta_x,\theta_y\rangle)
\]
Let us see how we can bound these two integral terms. We begin, for the sake of explanation if nothing else, with the simpler case where $u_0$ has compact support.
\begin{lem}\label{lem:bound_derivative_N_when_u_0_compact_support}
	Let $u_0\in L^1(\H^d)$ have compact support. Then there exist a positive constant $C$, depending on the diameter of the support of $u_0$, such that
	\[
	\mathcal{I} \leq C\, t^{-1}e^{-\frac{5}{2}\left(d+\frac{1}{2}\right)t}\|u_0\|_{L^1(\H^d)},
	\]
	and
	\[
	\mathcal{II} \leq C\, t^{-1}e^{-\frac{3}{2}\left(d-\frac{1}{2}\right)t}\|u_0\|_{L^1(\H^d)}.
	\]
\end{lem}
\begin{proof}
	Let us begin with integral $\mathcal{I}$. From~\eqref{eq:bound_heat_kernel_in_terms_of_h} we obtain
	\[
	\mathcal{I} \leq  \int_0^\infty \int_{\mathbb{S}^{d-1}}\int_{0}^\infty u_0(r_y, \theta_y)\, \sinh^{d+1}(r_y)\,\sin^2(\vartheta)\, h_{d+4}(t,\lambda(x,y))\, \sinh^{d-1}(r_x)\,\d r_y\,\d \theta_y  \, \d r_x.
	\]
	Now, let $a>0$. Then, using $\lambda(x,y)\geq |r_x-r_y|$, we get
	\[
	\begin{aligned}
		h_{d+4}(t,\lambda)\, \sinh^{d-1}(r_x)&= (4\pi t)^{-\frac{d+4}{2}}(1+\lambda+t)^{\frac{d+1}{2}}(1+\lambda)e^{-\frac{(\lambda+(d+3)t)^2}{4t}}\, \sinh^{d-1}(r_x)\\[10pt]
		&\lesssim t^{-\frac{3}{2}}\lambda^{\frac{d+3}{2}}e^{-\frac{3}{2}\lambda}\, e^{-\frac{1}{4t}f(r_x, r_y, t)}
	\end{aligned}
	\]
	with
	\[
	f(r_x, r_y, t):=|r_x-r_y|^2+2d\, t \,|r_x-r_y| - 4(d-1)t\, r_x + (d+3)^2t^2.
	\]
	But using that $r_x\leq |r_x-r_y|+r_y$ we get
	\[
	\begin{aligned}
		f(r_x, r_y, t)&\geq |r_x-r_y|^2-2(d-2)\, t \,|r_x-r_y| - 4(d-1)t\, r_y + (d+3)^2t^2\\[10pt]
		&= \big[|r_x-r_y| - (d-2)t\big]^2-4(d-1)t\, r_y -\big[(d-2)^2-(d+3)^2\big]t^2,
	\end{aligned}
	\]
	yielding
	\[
	h_{d+4}(t,\lambda)\, \sinh^{d-1}(r_x)\lesssim t^{-\frac{3}{2}}e^{-\frac{(|r_x-r_y| - (d-2)t)^2}{4t}}\cdot e^{(d-1)r_y}\cdot e^{\frac{(d-2)^2-(d+3)^2}{4}t}.
	\]
	Since
	\[
	\int_0^\infty e^{-\frac{(|r_x-r_y| - (d-2)t)^2}{4t}}\ \d r_x\lesssim \sqrt{t}\quad\text{and}\quad \sin^2(\vartheta)\leq 1
	\]
	we readily obtain
	\[
	\mathcal{I} \lesssim t^{-1}e^{\frac{(d-2)^2-(d+3)^2}{4}t}\int_0^\infty \int_{\mathbb{S}^{d-1}} u_0(r_y, \theta_y)\, \sinh^{d+1}(r_y)\, e^{(d-1)r_y}\,  \d r_y\,\d \theta_y.
	\]
	Using now that the support of $u_0$ is bounded, it is easy to conclude that
	\[
	\mathcal{I} \leq C\, t^{-1}e^{\frac{(d-2)^2-(d+3)^2}{4}t}\|u_0\|_{L^1(\H^d)},
	\]
	where $C$ is a constant that depends on the size of the support and increases as the diameter of the support increases. A similar computation, this time with $h_{d+2}$, yields
	\[
	\mathcal{II} \leq C\, t^{-1}e^{\frac{(d-2)^2-(d+1)^2}{4}t}\|u_0\|_{L^1(\H^d)}.
	\]
\end{proof}
As a consequence of~\eqref{eq:formula_for_derivative_N_in_time}, we readily obtain
\[
|\partial_\tau \mathcal{N}(\tau,\theta_x)|\lesssim e^{-\frac{1}{4}(2d-3)t}.
\]
However, asking for compact support of the initial datum is to constraining for our purposes. Let us see how to get rid of this condition.

We need for the next results sharper estimates for $\lambda(x,y)$. The simplest estimate is $\lambda(x,y)\geq |r_x-r_y|$, attained with equality whenever $\langle\theta_x,\theta_y\rangle = 1$. Next, in the case where $\cos(\vartheta) = \langle\theta_x,\theta_y\rangle \geq 0$ we have
\[
\cosh(r_x)\cosh(r_y)-\sinh(r_x)\sinh(r_y)\cdot \cos(\vartheta) \geq \frac{e^{r_x+r_y}}{4}(1-\cos(\vartheta)) = \frac{e^{r_x+r_y}}{2}\sin^2(\vartheta)
\]
and therefore, using $\text{arccosh}(z)=\ln(z+\sqrt{z^2-1})\geq \ln(z)$ for any $z\geq 0$ we readily obtain
\begin{equation}\label{eq:inferior_bound_lambda_positive_cosine}
	\lambda(x,y)\geq r_x+r_y+\ln\left(\frac{1}{2}\sin^2(\vartheta)\right)\quad\text{if}\quad \vartheta\in \left[-\frac{\pi}{2},\ \frac{\pi}{2}\right].
\end{equation}
Finally, when $\cos(\vartheta) \leq 0$ we have
\[
\cosh(r_x)\cosh(r_y)-\sinh(r_x)\sinh(r_y)\cdot \cos(\vartheta) \geq \frac{e^{r_x+r_y}}{4}
\]
and so
\begin{equation}\label{eq:inferior_bound_lambda_negative_cosine}
	\lambda(x,y)\geq r_x+r_y-\ln\left(4\right)\quad\text{if}\quad \vartheta\in \left[\frac{\pi}{2},\ \frac{3\pi}{2}\right].
\end{equation}

We will also use the following proposition and lemma.
\begin{prp}\label{prp:integral_in_the_sphere_negative_sine}
	Let $\theta_x\in\S^{d-1}$ be fixed and $s\in\R$. Then
	\[
	\int_{\S^{d-1}}|\sin(\vartheta/2)|^s\, \d \theta_y<\infty\quad\text{for all}\quad s>1-d
	\]
	where $\vartheta=\arccos(\langle\theta_x, \theta_y\rangle)$.
\end{prp}

\begin{lem}\label{lem:growth_condition_u_0}
	Let $u_0\in L^1(\H^d)$ such that there exists an $\varepsilon\in(0,1)$ such that
	\begin{equation}\label{eq:condition_u_0_radial}
		\sup\limits_{\theta_y\in\S^{d-1}}\left\{\int_C^\infty u_0(\theta_y, r_y)\sinh^{d-1+\varepsilon}(r_y)\, \d r_y\right\}\leq K_\varepsilon(u_0)
	\end{equation}
	for a certain constant $K_\varepsilon(u_0)$. Then, by definition, $\mathcal{N}(0,\theta_x)\in L^\infty(\S^{d-1})$ and, moreover,
	\[
	\mathcal{I}\lesssim t^{-1}e^{-(3d+\tilde\varepsilon(d))t} K_\varepsilon(u_0)
	\]
	and
	\[
	\mathcal{II}\lesssim t^{-1}e^{-(2d-2+\tilde\varepsilon(d))t} K_\varepsilon(u_0),
	\]
	with $\tilde\varepsilon(d)=((6-2d)\varepsilon-\varepsilon^2)/4$.
\end{lem}
\begin{proof}

	As a corollary of Lemma~\ref{lem:bound_derivative_N_when_u_0_compact_support}, we only have to worry about the $r_y$ in the set $(C, \infty)$ for any $C>0$, since the integral over $r_y$ in the set $r_y\in[0, C]$ can be bounded by a similar argument.
	
	So, we focus on the set $r_y\in (C, \infty)$ for some $C>0$. Next, given $\theta_x$ and recalling $\vartheta=\arccos(\langle\theta_x, \theta_y\rangle)$, we divide the integral over $\theta_y\in\S^{d-1}$ in two sets,
	\[
	S_1:=\left\{\theta_y\in \S^{d-1}\text{ such that }\vartheta\in \left[-\frac{\pi}{2},\ \frac{\pi}{2}\right] \right\}.
	\]
	and
	\[
	S_2:=\left\{\theta_y\in \S^{d-1}\text{ such that }\vartheta\in \left[\frac{\pi}{2},\ \frac{3\pi}{2}\right] \right\}.
	\]
	
	First, for the integral in $\mathcal{I}$ over $\int_C^\infty \int_{S_1}\int_{0}^\infty\,\d r_y\,\d \theta_y  \, \d r_x$. In this set we use the bound $\lambda(x,y)\geq |r_x-r_y|$  and the one from~\eqref{eq:inferior_bound_lambda_positive_cosine} to see that
	\[
	h_{d+4}(t,\lambda)\, \sinh^{d-1}(r_x)\lesssim t^{-\frac{3}{2}}\lambda^{\frac{d+3}{2}}e^{-\frac{2+\varepsilon}{2}\lambda}\, e^{-\frac{1}{4t}f(r_x, r_y, t)}
	\]
	with
	\[
	f(r_x, r_y, t):=|r_x-r_y|^2+2(d+1-\varepsilon)\, t \,\left(r_x+r_y+\ln\left(\frac{1}{2}\sin^2(\vartheta/2)\right)\right) - 4(d-1)t\, r_x + (d+3)^2t^2.
	\]
	It is not hard to see that $f$ can be written as
	\[
	\begin{aligned}
		f(r_x, r_y, t):=& \left[(r_x-r_y) - (d-3+\varepsilon)t\right]^2 -\left[(d-3+\varepsilon)^2-(d+3)^2\right]t^2 + 4\, t \, r_y \\[10pt]
		& + 2(d+1-\varepsilon)t\ln\left(\frac{1}{2}\sin^2(\vartheta/2)\right)
	\end{aligned}
	\]
	This yields
	\[
	h_{d+4}(t,\lambda)\, \sinh^{d-1}(r_x)\lesssim t^{-\frac{3}{2}}e^{-\frac{((r_x-r_y) - (d-2)t)^2}{4t}}\cdot e^{-r_y}\cdot e^{\frac{(d-3+\varepsilon)^2-(d+3)^2}{4}t}\cdot |\sin(\vartheta/2)|^{-(d+1-\varepsilon)},
	\]
	which in turn provides, using Proposition~\ref{prp:integral_in_the_sphere_negative_sine} in the final step,
	\[
	\begin{aligned}
		&\int_C^\infty \int_{S_1}\int_{0}^\infty\, (\cdots)\, \d r_y\,\d \theta_y  \, \d r_x\\[10pt]
		&\quad\lesssim t^{-1}e^{\frac{\left(d-3+\varepsilon \right)^2-(d+3)^2}{4}t} \int_C^\infty \int_{S_1}  u_0(r_y, \theta_y)\, \sinh^{d-1+\varepsilon}(r_y)\, \sin^2(\vartheta)\, |\sin(\vartheta/2)|^{-(d+1-\varepsilon)}  \d r_y\,\d \theta_y\\[10pt]
		&\quad\lesssim t^{-1}e^{-(3d+\tilde\varepsilon(d))t} \int_C^\infty \int_{S_1}  u_0(r_y, \theta_y)\, \sinh^{d-1+\varepsilon}(r_y)\, |\sin(\vartheta/2)|^{1-d+\varepsilon}\,  \d r_y\,\d \theta_y\\[10pt]
		&\quad = t^{-1}e^{-(3d+\tilde\varepsilon(d))t} \int_{S_1}    |\sin(\vartheta/2)|^{1-d+\varepsilon}\left( \int_C^\infty  u_0(r_y, \theta_y)\,\sinh^{d-1+\varepsilon}(r_y)\, \d r_y\right)\,\d \theta_y\\[10pt]
		&\quad\lesssim t^{-1}e^{-(3d+\tilde\varepsilon(d))t} K(u_0),
	\end{aligned}
	\]
	with $\tilde\varepsilon(d)=((6-2d)\varepsilon-\varepsilon^2)/4$. Note that $\tilde{\varepsilon}(2)>0$ for all $\varepsilon\in (0,1)$.	
	
	The case for the integral in $\mathcal{I}$ over $\int_C^\infty \int_{S_2}\int_{0}^\infty\,\d r_y\,\d \theta_y  \, \d r_x$ is simpler since we can use is a similar way the simpler bounds  $\lambda(x,y)\geq |r_x-r_y|$ and~\eqref{eq:inferior_bound_lambda_negative_cosine}. In total, it is easy to see that
	\[
	\mathcal{I}\lesssim t^{-1}e^{-(3d+\tilde\varepsilon(d))t} K_\varepsilon(u_0).
	\]	
	The bound
	\[
	\mathcal{II}\lesssim t^{-1}e^{-(2d-2+\tilde\varepsilon(d))t} K_\varepsilon(u_0)
	\]
	comes analogously
\end{proof}

Note that the condition~\eqref{eq:condition_u_0_radial} is not entirely equivalent to asking
\[
\mathcal{M}_\varepsilon(u_0):= \int_0^\infty \int_{\S^{d-1}} u_0(r_y,\theta_y)\sinh^{d-1+\varepsilon}(r_y)\, \d \theta_y\, \d r_y <\infty,
\]
which can be understood as an $\varepsilon$-moment of the initial datum. It is, however, sufficient, since it is clear that
\[
\mathcal{M}_\varepsilon(u_0)\leq |\S^{d-1}|\cdot K_\varepsilon(u_0).
\]

Again, as a consequence of~\eqref{eq:formula_for_derivative_N_in_time} and the previous lemma, we readily obtain
\[
|\partial_\tau \mathcal{N}(\tau,\theta_x)|\lesssim e^{-(d-2+\tilde{\varepsilon}(d))\, t}\sim e^{-(d-2+\tilde{\varepsilon}(d))\, e^\tau}.
\]
Choosing $\varepsilon$ small enough we readily obtain
\[
|\partial_\tau \mathcal{N}(\tau,\theta_x)|\lesssim e^{-\delta_d\, e^\tau}
\]
for a certain $\delta_d\in (0, d-2)$ for $d\geq 3$ and $\delta_d\in (0,1)$ for $d=2$. Therefore,
\[
|\mathcal{N}(\tau_1,\theta_x) - \mathcal{N}(\tau_0,\theta_x)| \leq \int_{\tau_0}^{\tau_1} e^{-\delta_d e^\tau}\ \d \tau < \int_{\tau_0}^{\tau_1} e^{\tau-\delta_d e^\tau}\ \d \tau = \frac{e^{-\delta_d e^{\tau_0}}-e^{-\delta_d e^{\tau_1}}}{\delta_d}.
\]
Thus, for any sequence $\{\tau_k\}\to\infty$ the corresponding sequence $\{\mathcal{N}(\tau_k, \theta_x)\}$ is a Cauchy sequence. The limit
\[
\mathcal{N}_\infty(\theta_x):= \lim\limits_{\tau\to\infty}  \mathcal{N}(\tau, \theta_x)
\]
exists and satisfies
\[
|\mathcal{N}_\infty(\theta_x) - \mathcal{N}(\tau, \theta_x)|\lesssim \frac{e^{-\delta_d e^{\tau}}}{\delta_d}.
\]
Moreover, since $\mathcal{N}(0,\theta_x)\in L^\infty(\S^{d-1})$, we should have $\mathcal{N}(\tau,\theta_x)\in L^\infty(\S^{d-1})$ for all $\tau\geq 0$.

We summarize the previous computations in this lemma.

\begin{lem}\label{lem:bound_derivative_N_general datum}
	Let $u_0\in L^1(\H^d)$ such that condition~\eqref{eq:condition_u_0_radial} is satisfied. Then for all $\theta_x\in\mathbb{S}^{d-1}$ the function $\mathcal{N}(\tau,\theta_x)\in L^\infty(\S^{d-1})$ converges, as time grows to infinity, to a function $\mathcal{N}_\infty(\theta_x)\in L^\infty(\mathbb{S}^{d-1})$. Moreover, for a certain $\delta_d>0$ depending only on the dimension $d$,
	\[
	|\partial_\tau \mathcal{N}(\tau,\theta_x)|\lesssim e^{-\delta_d e^\tau}
	\]
	and
	\[
	|\mathcal{N}_\infty(\theta_x) - \mathcal{N}(\tau, \theta_x)|\lesssim \frac{e^{-\delta_d e^{\tau}}}{\delta_d}.
	\]
\end{lem}

\subsection{Convergence with rates}\label{sect:convergence_with_rates}

Using formula~\eqref{eq:entropy_with_N} and Lemmata~\ref{lem:bounds_N} and~\ref{lem:bound_derivative_N_general datum} we arrive to
\[
\partial_\tau H(\tau) \leq  k_d\ \|u_0\|_{L^1(\H^d)}\ e^{\tau-m_d e^\tau}+ |\mathbb{S}^{d-1}|\ \ln(C_2(u_0))\ \frac{e^{-\delta_d e^{\tau}}}{\delta_d}  - \lambda H(\tau),
\]
Again as in Section~\ref{sect:elliptic} we can conclude from here the existence of a value $H^*(v_0)$ depending only on the initial datum and the dimension $d$ such that
\[
H(\tau)\leq H^*(v_0)e^{-\lambda \tau},
\]
hence
\[
\|v- \mathcal{N}V\|_{L^1}\leq \sqrt{H^*(v_0)}(t+1)^{-\frac{\lambda}{2}}
\]

We recover the coordinate system $x=(\theta, r)$. The previous result can be translated to
\begin{equation}\label{eq:convergence_N(t)}
	\| u(t,x) - C(t)\ \mathcal{N}(t,\theta)\ \Gamma(t+1, r+(d-1)(t+1)) \|_{L^1(\mathbb{H}^d)}  \leq \sqrt{H^*(u_0)}\cdot (t+1)^{-\frac{1}{2}},
\end{equation}
where $C(t)$ is chosen so that
\[
C(t)\int_0^\infty \Gamma(t+1, r+(d-1)t) \ \d \mu(r)= 1.
\]
Note that, by Lemma~\ref{lem:constant_C(t)},
\[
C_\infty = \frac{2^{d-2}}{\sqrt{\pi}}\leq C(t)\leq \frac{2^{d-2}}{\sqrt{\pi}} + O(e^{-m_d(t+1)}).
\]

Let us now compute the value of $\mathcal{N}_\infty$. We define, for readability,
	\[
	S_{d-1}:=|\S^{d-1}|.
	\]
    Then we have:
\begin{lem}\label{lem:value_N_infty}
	Let $u_0\in L^1(\H^d)$ such that condition~\eqref{eq:condition_u_0_radial} is satisfied  and $\Phi \in L^1(\mathbb{S}^{d-1})$, defined as in~\eqref{eq:def_Phi} . Then $S_{d-1}\cdot \mathcal{N}_\infty(\theta) = \Phi(\theta)$ a.e. in $\S^{d-1}$.
\end{lem}
\begin{proof}
    On the one hand, since\begin{equation}\label{eq:mass_kernel_equal_1}
		\int_{\H^d} P_d(t,x)\ \d \mu(x) = S_{d-1} \int_0^\infty G_d(t,r)\ \d \mu(r)= 1
	\end{equation}
	we deduce
	\[
	C(t)\int_0^\infty \Gamma(t+1, r+(d-1)t) \ \d \mu(r)=S_{d-1} \int_0^\infty G_d(t,r)\ \d \mu(r)= 1.
	\]
	From Theorem~\ref{thm:L^1_radially_symmetric_asymptotic} we obtain
	\begin{equation}\label{eq:convergence_Gamma_SG}
		\| C(t)\ \Gamma(t+1, r+(d-1)t)-S_{d-1}\ G_d(t,r)  \|_{L^1([0,\infty]; \d \mu)}\lesssim (t+1)^{-\frac{1}{2}}.
	\end{equation}
	
	On the other hand, from~\eqref{eq:mass_kernel_equal_1}, we obtain
	\[
	\|\Phi - S_{d-1}\ \mathcal{N}_\infty \|_{L^1(\S^{d-1})} = S_{d-1}\cdot \|\Phi\ G_d - S_{d-1}\ N_\infty G_d   \|
	\]
	and therefore, as a consequence of the triangular inequality,
	\[
	\begin{aligned}
		&\|\Phi - S_{d-1}\ \mathcal{N}_\infty \|_{L^1(\S^{d-1})}\\[10pt]
		&\leq S_{d-1}\left( \underbrace{\|C(t)\mathcal{N}_\infty \Gamma - S_{d-1}\ \mathcal{N}_\infty G_d\|_{L^1(\H^{d-1})}}_{A_1(t)} + \underbrace{\|u   -  C(t)\mathcal{N}_\infty \Gamma\|_{L^1(\H^{d-1})}}_{A_2(t)} + \underbrace{\|u - \Phi G_d\|_{L^1(\H^{d-1})}}_{A_3(t)} \right).
	\end{aligned}
	\]
	From~\eqref{eq:convergence_N(t)} we have $A_2(t)\lesssim t^{-\frac{1}{2}}$, and from Theorem~\ref{thm:convergence_without_rate}, $A_3(t)\to 0$ as $t$ grows. Finally, using~\eqref{eq:convergence_Gamma_SG},
	\[
	A_1(t) = \|\mathcal{N}_\infty\|_{L^1(\S^{d-1})}\cdot \|C(t) \Gamma - S_{d-1} G_d\|_{L^1([0,\infty], \d\mu)}\lesssim \|\mathcal{N}_\infty\|_{L^1(\S^{d-1})} (t+1)^{-\frac{1}{2}}.
	\]
	Therefore, there exists a function $f(t)$ such that
	\[
	\|\Phi - S_{d-1}\ \mathcal{N}_\infty \|_{L^1(\S^{d-1})}\leq f(t)\to 0.
	\]
	Since $\Phi$ and $\mathcal{N}_\infty$ do not depend on time, we conclude that
	\[
	\|\Phi - S_{d-1}\ \mathcal{N}_\infty \|_{L^1(\S^{d-1})} = 0,
	\]
	and the result follows.
	
\end{proof}

Putting together Lemmata~\ref{lem:bound_derivative_N_general datum} and~\ref{lem:value_N_infty}, and inequality~\eqref{eq:convergence_N(t)}, we obtain the following Theorem. The proof is left for the reader as an exercise.

\begin{thm}\label{thm:integrable_L^1_Lînfty_convergence}
	Let $x\in\H^d$ be given by the coordinates $x=(r_x,\theta_x)$. Let $u_0\in L^1(\H^d)$ such that condition~\eqref{eq:condition_u_0_radial} is satisfied  and $\Phi \in L^1(\mathbb{S}^{d-1})$, defined as in~\eqref{eq:def_Phi}. Then,
	\[
	\| u(t,x) - C(t)\ \mathcal{N}(t,\theta)\ \Gamma(t+1, r_x+(d-1)(t+1)) \|_{L^1(\mathbb{H}^d)}  \leq \sqrt{H^*(u_0)}\cdot (t+1)^{-\frac{1}{2}}.
	\]
	Moreover, there exists a time $t_0$ such that, for all $t>t_0$,
	\[
	\| u(t,x) - C_\infty\ (S_{d-1})^{-1}\ \Phi(\theta_x)\ \Gamma(t+1, r_x+(d-1)(t+1)) \|_{L^1(\mathbb{H}^d)}  \lesssim t^{-\frac{1}{2}},
	\]
	where
	\[
	C_\infty = \frac{2^{d-2}}{\sqrt{\pi}},\quad S_{d-1} = |\S^{d-1}|.
	\]
	If we want to express convergence in terms of the heat kernel $G_d$,
	\[
	\| u(t,x) - \Phi(\theta_x)\ G_d(t,r_x) \|_{L^1(\mathbb{H}^d)}  \lesssim t^{-\frac{1}{2}}\quad \text{for all}\quad t\geq t_0.
	\]
\end{thm}

\begin{rem}
	Following Remark~\ref{rem:varphi_continuous_phi_bounded}, conditions $\Phi \in L^1(\mathbb{S}^{d-1})$ and \eqref{eq:condition_u_0_radial} are satisfied when the initial data $u_0$ has compact support.
\end{rem}

\begin{rem}
	We do not claim whether condition~\eqref{eq:condition_u_0_radial} is sufficient or not for obtaining $\Phi \in L^1(\mathbb{S}^{d-1})$. It is, at least for small $\varepsilon$, not sufficient to obtain $\Phi \in L^\infty(\mathbb{S}^{d-1})$ (check conditions~\eqref{eq:condition_papgeorgiiou} and~\eqref{eq:condition_exponential_initial_datum} for a sufficient one, noticing the difference between $\d \mu(y)$ and $\d r_y$), but indeed we could have $\Phi \in L^1(\mathbb{S}^{d-1})$ while $\Phi \not\in L^\infty(\mathbb{S}^{d-1})$; this question is beyond the scope of this article. In the case where we do not know if condition~\eqref{eq:condition_u_0_radial} is satisfied, then we may not have convergence of $\mathcal{N}(\tau, \cdot)$. On the other hand, in the case where we do not know if $\Phi \in L^1(\mathbb{S}^{d-1})$ then we do not know how to identify the limit $\mathcal{N}_\infty$, and therefore the best convergence that we obtain is
	\[
	\| u(t,x) - C_\infty\ \mathcal{N}_\infty(\theta_x)\ \Gamma(t+1, r_x+(d-1)(t+1)) \|_{L^1(\mathbb{H}^d)}  \lesssim t^{-\frac{1}{2}}\quad \text{for all}\quad t\geq t_0
	\]
	for a certain unidentified function $\mathcal{N}_\infty$.
\end{rem}

\appendix

\section{Appendix}

\subsection{Construction of Fermi-coordinate systems based on hypersurfaces}\label{sect:construction_coordinates}

By a symmetry we mean the action of a subgroup of the \(\mathrm{Iso}(\mathbb{H}^d)\), the isometry group of hyperbolic space. There are three principal subgroups of \(\mathrm{Iso}(\mathbb{H}^d)\), corresponding to the three types of isometries in $\mathbb{H}^d$ with codimension one orbits: elliptic, parabolic, and hyperbolic. Each subgroup naturally gives rise to a distinct coordinate system. %To prepare for this, we first describe a general procedure to construct a coordinate system based on the distance to a hypersurface.
Note that in these cases the foliation of $\H^d$ by parallel hypersurfaces \(\{\Sigma_r\}_{r \in I}\) coincides with the orbits of one of the principal subgroups of \(\mathrm{Iso}(\mathbb{H}^d)\).  Let us discuss now the general procedure to construct a coordinate system based on the distance to a hypersurface.

We begin by recalling the definition of the exponential map at a point. Given \(x \in \mathbb{H}^d\), the exponential map at \(x\) is the function  
\[
\exp_x : T_x \mathbb{H}^d \longrightarrow \mathbb{H}^d
\]
defined as follows: for \(v \in T_x \mathbb{H}^d\), let \(\gamma_{x,v}: \mathbb{R} \to \mathbb{H}^d\) denote the unique geodesic with initial conditions \(\gamma_{x,v}(0) = x\) and \(\gamma_{x,v}'(0) = v\). Then  
\[
\exp_x(v) := \gamma_{x,v}(1).
\]
In other words, \(\exp_x\) maps a tangent vector \(v\) at \(x\) to the point obtained by moving from \(x\) along the geodesic in the direction of \(v\) for a distance \(\|v\|\). Here, $\| . \|$ is the norm induced by the hyperbolic metric.

By fixing a basis of \(T_x \mathbb{H}^d\), one may identify the tangent space \(T_x \mathbb{H}^d\) with \(\mathbb{R}^d\). This allows us to regard directions in the tangent space as Euclidean vectors, making the local geometry more accessible.  

Let now \(\Sigma \subset \mathbb{H}^d\) be a fixed smooth, embedded, closed hypersurface, and suppose that there exists \(N\) a smooth normal vector field along \(\Sigma\). Thus, $N$ is a vector field defined on $\Sigma$ such that \(N(x) \perp T_x \Sigma\) and $ \| N(x) \|=1$ for every \(x \in \Sigma\). We define the normal exponential map to \(\Sigma\) by  
\[
\begin{split}
	E : \mathbb{R} \times \Sigma &\longrightarrow \mathbb{H}^d \\
	(r,x) &\longmapsto \exp_x \big(r N(x)\big).
\end{split}
\]
Suppose there exists an open interval \(I \subset \mathbb{R}\), containing the origin, such that \(E\) is injective when restricted to \(\mathcal{U} := I \times \Sigma\). Then \(E|_{\mathcal{U}}\) defines a diffeomorphism between \(\mathcal{U}\) and an open subset \(\mathcal{V} \subset \mathbb{H}^d\), foliated by hypersurfaces parallel to \(\Sigma\).

Note that since $\Sigma$ is closed and embedded we have that it divides $\mathcal{V}$ into two open connected components: $\mathcal{V} \setminus \Sigma = \mathcal{V}_+ \cup \mathcal{V}_-$. In this context, we say that $N$ points towards $\mathcal{O} \subset \mathcal{V} \setminus \Sigma$ if for any $x \in \Sigma$ there exists $\varepsilon >0$ such that $\textup{exp}_x (t N(x)) \in \mathcal{O}$ for all $t \in (0, \varepsilon)$. If we suppose that $N$ points towards $\mathcal{V}_+$ then the parameter $r$ in the normal exponential map can be defined as the \textit{signed distance to} $\Sigma$, which is given by
\begin{equation}\label{signedDistance}\notag
	r(x):=\begin{cases}
		\, \, \,  \, \textup{dist} (x, \Sigma), \quad & \forall x\in\mathcal{V}_+,\\
		-\textup{dist} (x, \Sigma), & \forall x \in \mathcal{V}_-.
	\end{cases}
\end{equation}
Thus, the above construction yields a coordinate neighbourhood in which each point \(x \in \mathcal{V}\) is described by its signed distance \(r\) to \(\Sigma\), together with the point \(x \in \Sigma\) where this distance is attained. We say then that $\mathcal{U}$ is a \textit{coordinate neighbourhood centred at }$\Sigma$ \textit{with respect to} $N$. In these coordinates, the hyperbolic metric takes the form  
\[
g_{\mathbb{H}^d} = dr^2 + g_{\Sigma_r},
\]
where \(g_{\Sigma_r}\) denotes the hyperbolic metric restricted to the parallel hypersurface \(\Sigma_r := E(r,\Sigma)\).

\subsection{Logarithmic Sobolev Inequalities}\label{sect:sobolev}

In this section we will demonstrate the series of one-dimensional logarithmic Sobolev Inequalities that we have been employing in this article. Since they appear naturally in the study of solutions of problem~\eqref{eq:heatEquationHyperbolic}, we consider these to be the natural log-Sobolev Inequalities in the hyperbolic space $\H^d$.

Let $I\subseteq \R$. For a positive, integrable function
$F \: I \to \R$ we define $\lambdaL \equiv \lambdaL(F) \geq 0$ as the best constant in the
inequality
\begin{equation*}
	\lambdaL \int_{\Omega}g
	\log\frac{g}{F}\leq \int_{\Omega}g\left|\nabla
	\log\frac{g}{F}\right|^2
\end{equation*}
%\begin{equation*}
%	\lambdaL H(g\,|\, F)\leq \int_{\Omega}g\left|\nabla
%	\log\frac{g}{F}\right|^2
%\end{equation*}
for all positive $g \in L^1(I)$ with
$\int_\Omega g = \int_\Omega F$ (understanding the right-hand side to
be equal to $+\infty$ whenever $\log \frac{g}{F}$ does not have a weak
gradient, or when its weak gradient is defined but the integral on the
right-hand side is infinite). We say that $F$ satisfies a logarithmic
Sobolev Inequality (also known as Gross Inequality) when $\lambdaL(F) > 0$. Notice that integral in the left-hand side of the previous inequality is nothing but the entropy of $g$ relative to $F$, that is, $H(g|F)$.

%For later use, we also
%denote by $\lambdaP \equiv \lambdaP(F) \geq 0$ the best constant in
%the Poincaré inequality
%\begin{equation*}
%	\lambdaP \int_\Omega \left( 1 - \frac{g}{F} \right)^2 F
%	\leq \int_{\Omega} \left|\nabla \frac{g}{F}\right|^2 F,
%\end{equation*}
%for all $g \in L^1(\Omega)$ with $\int_\Omega g = \int_\Omega F$ (with
%the same understanding as before: the right-hand side is equal to
%$+\infty$ unless $g/F$ has a weak gradient in $L^2(F)$),

There are several results that will be useful for us when estimating
the logarithmic Sobolev constant $\lambdaL$. The first one is a
consequence of the well-known curvature-dimension condition
\cite{Ane-Blachere-Chafa-Fougeres-Gentil-Malrieu-Roberto-Scheffer-2000,Bakry2014}, and this particular statement can be
obtained from the theory presented in \cite[Section 5]{Ane-Blachere-Chafa-Fougeres-Gentil-Malrieu-Roberto-Scheffer-2000} and
the proof of \textit{Corollaire 5.5.2} therein:
\begin{lem}
	\label{lem:curvature-logsob}
	Let $r_0 \in\R$,  $F \: (r_0,+\infty) \to (0,+\infty)$ be a positive,
	integrable function of the form
	\begin{equation*}
		F(x) = C e^{-\Phi(x)},
		\qquad x > r_0,
	\end{equation*}
	with $C > 0$ and $\Phi \: (r_0,+\infty) \to \R$ a convex,
	$\mathcal{C}^2$ function with $\Phi''(x) \geq \Lambda$ for all $x >
	r_0$. Then the logarithmic Sobolev inequality
	\begin{equation*}
		2 \Lambda\cdot H(g \,|\, F)
		\leq \int_{r_0}^\infty g \left| \nabla  \log\frac{g}{F} \right|^2
	\end{equation*}
	holds for all nonnegative $g \in L^1(r_0,+\infty)$ with $\int_\Omega g = \int_\Omega F$.
\end{lem}

We also cite a well known result on perturbation of these
inequalities by \cite{Holley1987}:
\begin{lem}
	\label{lem:HolleyStroock}
	Let $I \subseteq \R$ be an open set, $F \: I \to \R$ a
	positive, integrable function which satisfies a logarithmic Sobolev
	inequality with constant $\lambda_L$. Let $A \: I \to \R$ be a
	measurable function such that $|A|$ is bounded. Then the function
	\begin{equation*}
		\widetilde{F}(x) = F(x) e^{-A(x)}, \qquad x \in I,
	\end{equation*}
	also satisfies a logarithmic Sobolev inequality with constant
	$\lambda_L e^{\operatorname{osc}(A)}$, where
	$\operatorname{osc}(A) := \sup A - \inf A$.
\end{lem}

We present now the result needed for our study.

\begin{cor}\label{cor:log_sobolev}
	Fix $d\geq 2$, $\tau>0$ and let $r(\tau, \rho):= (d-1)e^\tau + e^\frac{\tau}{2}\rho$, $\rho_0(\tau) = -(d-1)e^\frac{\tau}{2}$ and $C(\tau)= C+o(\tau)$. Then, the functions
	\[
	V_1(\tau, \rho): = C(\tau)\cdot \sinh^{d-1}\big(r(\tau, \rho)\big)\cdot e^{- \frac{\left(\rho + 2(d-1)e^\frac{\tau}{2}\right)^2}{4}},
	\]
	and
	\[
	V_2(\tau, \rho): = C(\tau)\cdot \cosh^{d-1}\big(r(\tau, \rho)\big)\cdot e^{- \frac{\left(\rho + 2(d-1)e^\frac{\tau}{2}\right)^2}{4}},
	\]
	satisfy a log-Sobolev Inequality with constant $\lambdaL= 1$ on the interval $I=(\rho_0(\tau), \infty)$. That is, for $i=1, 2$,
	\[
	\int_{\rho_0(\tau)}^\infty v \cdot\ln\left(\frac{v}{V_i}\right) \leq \int_{\rho_0(\tau)}^\infty v \cdot\left|\partial_\rho\ln\left(\frac{v}{V_i}\right)\right|^2,
	\]
	for all nonnegative $v$ such that $\int_I v = \int_I V_i$
\end{cor}
\begin{proof}
	The inequality for $V_1$ is a direct consequence of Lemma~\ref{lem:curvature-logsob}, while the one for $V_2$ comes from Lemma~\ref{lem:HolleyStroock} after checking that
	\[
	\cosh^{d-1}(r(\tau,\rho))\in \left[2^{1-d}e^{(d-1)^2e^\tau + (d-1)e^\frac{\tau}{2}\rho} ,\ e^{(d-1)^2e^\tau + (d-1)e^\frac{\tau}{2}\rho}\right]
	\]
	and thus $V_2$ is trapped between two gaussian profiles, for which the log-Sobolev inequality with constant $\lambdaL=1$ is known.
\end{proof}

\subsection{Estimates for the radially symmetric transient equilibrium}\label{sec:UsefulBounds}
In this section we collect some estimates for the transient equilibrium obtained in Section \ref{sect:elliptic}. We recall that this is the function
\[
V(t,r) = C(t)\cdot \sinh^{d-1}(r)\cdot e^{- \frac{\left(r + (d-1)(t+1)\right)^2}{4(t+1)}}.
\]
where we choose the function $C(t)$ such that
\[
\mathcal{M}_{\mathcal{R}} = \int_{\rho_0(\tau)}^\infty V(\tau, \rho)\ \d \rho,
\]
where
\[
t(\tau)=e^\tau-1,\quad r(\tau,\rho)= (d-1)e^\tau  + e^\frac{\tau}{2}\rho.
\]

\begin{lem}\label{lem:constant_C(t)}
	There exists a constant $k_d$ depending only on the dimension $d$ such that
	\begin{equation*}
		-k_d\  e^{-m_d (t+1)}\leq \frac{\partial_t C(t)}{C(t)}\leq k_d\  e^{-m_d (t+1)},
	\end{equation*}
	where
	\[
	m_d := \min \left(d-1, \frac{(d-1)^2}{16}\right).
	\]
	Moreover,
	\begin{equation*}
		\frac{2^{d-2}\mathcal{M}_{\mathcal{R}}}{\sqrt{\pi} }\leq C(t)\leq \frac{2^{d-2}\mathcal{M}_{\mathcal{R}}}{(1-e^{-(d-1)(t+1)})^{d-1}\left[\sqrt{\pi} - \frac{1}{\sqrt{\lambda_1 (t+1)}}e^{-\frac{\lambda_1 (t+1)}{4}}\right]}
	\end{equation*}
	and, as a consequence,
	\[
	C_\infty:=\lim\limits_{t\to\infty} C(t) = \frac{2^{d-2}\mathcal{M}_{\mathcal{R}}}{\sqrt{\pi} }.
	\]
\end{lem}
\begin{proof}
	For the sake of simplicity, we will consider $t(\tau)=e^\tau$ in the computations instead of the real change of variables $t(\tau)=e^\tau-1$. This does not affect the results, it is just a translation in time; instead of $t\in [0,\infty)$ we will work on the time interval $t\in[1, \infty)$, and then a map $t\to t+1$ will undo de translation.
	
	In order to estimate the function $C(t)$ we notice that $V$ must have constant mass, since a solution of equation~\eqref{eq:main_spherical_coordinates_symmetric_datum} does so. Therefore,
	\begin{equation}\label{eq:equality_C_radial_solutions}
		\begin{aligned}
			\mathcal{M}_{\mathcal{R}} = \int_{\rho_0(\tau)}^\infty V(\tau, \rho)\ \d \rho &= \frac{C(t)}{\sqrt{t}}\int_{0}^\infty \sinh^{d-1}(r)\cdot e^{- \frac{\left(r + (d-1)t\right)^2}{4t}}\ \d r
		\end{aligned}
	\end{equation}	
	for all $t\geq 1$. Using implicit derivation in the equation in~\eqref{eq:equality_C_radial_solutions} it is possible to estimate $\partial_t C(t)/C(t)$. It can be verified, after some computations, that
	\begin{equation}\label{eq:quotient_C(t)}
		\frac{\partial_t C(t)}{C(t)} = \frac{(d-1)^2t+2}{4t} - \frac{1}{4t^2}\cdot \frac{\mathcal{I}}{\mathcal{II}},
	\end{equation}
	where
	\[
	\mathcal{I}:= \int_0^\infty r^2 (1-e^{-2r})^{d-1}\ e^{- \frac{\left(r - (d-1)t\right)^2}{4t}}\ \d r,\quad  \mathcal{II}:= \int_0^\infty (1-e^{-2r})^{d-1}\ e^{- \frac{\left(r - (d-1)t\right)^2}{4t}}\ \d r.
	\]
	
	Let us begin with a lower estimate  for ~\eqref{eq:quotient_C(t)}.
	%	
	%	First, we estimate the quotient~\eqref{eq:quotient_C(t)} on a compact interval $[1, T^*_d]$, for a $T^*_d\geq 1$ to be chosen later. Since the quotient is finite at time $t=1$ and the functions of time involved in its formulas are continuous , it is clear that there must exist a constant $K(T^*_d)$ such that
	%	\[
	%	-K(T^*_d)\leq \frac{\partial_t C(t)}{C(t)}\leq  K(T^*_d) \quad\text{for all}\quad t\in [1, T^*_d].
	%	\]
	%	
	%	Now, for times bigger than $T^*_d$.
	We estimate each integral separately. First, with a change of variables and using that $1-e^{-2r}<1$ ,
	\[
	\begin{aligned}
		\mathcal{I} &\leq  2\sqrt{t}\int_{-\sqrt{\lambda_1 t}}^\infty \left(2\sqrt{t}z + (d-1)t\right)^2 \ e^{-z^2}\ \d z\\[10pt]
		& = \underbrace{8t^\frac{3}{2}\int_{-\sqrt{\lambda_1 t}}^\infty z^2 \ e^{-z^2}\ \d z}_{I_1} + \underbrace{8(d-1)t^2 \int_{-\sqrt{\lambda_1 t}}^\infty z \ e^{-z^2}\ \d z}_{I_2} + \underbrace{2 (d-1)^2 t^\frac{5}{2} \int_{-\sqrt{\lambda_1 t}}^\infty  e^{-z^2}\ \d z}_{I_3},
	\end{aligned}
	\]
	where $\lambda_1$ comes from~\eqref{eq:definition_bottom_spectrum_laplacian}. The integral $I_1$ can be bounded with an integration by parts and using that
	\[
	\int_{-\sqrt{\lambda_1 t}}^\infty e^{-z^2}\d z\leq  \int_{-\infty}^\infty e^{-z^2}\d z = \sqrt{\pi}
	\]
	by
	\[
	I_1\leq 2\sqrt{t}\left(2\sqrt{\pi}\ t - (d-1)t^\frac{3}{2}e^{-\lambda_1 t}  \right),
	\]
	while $I_2$ can be computed explicitly as
	\[
	I_2 = 4(d-1)t^2 e^{-\lambda_1 t}
	\]
	and
	\[
	I_3 \leq 2 (d-1)^2 \sqrt{\pi}\  t^\frac{5}{2}.
	\]
	In total,
	\[
	\mathcal{I}\leq 2\sqrt{t}\left((d-1)^2\sqrt{\pi}\ t^2  +2\sqrt{\pi}\ t  + (d-1)t^\frac{3}{2}e^{-\lambda_1 t} \right)
	\]
	Second, we compute for $\mathcal{II}$,
	\[
	\begin{aligned}
		\mathcal{II} &= 2\sqrt{t} \int_{-\sqrt{\lambda_1 t}}^{\infty} \left(1- e^{-2(2\sqrt{t}z+(d-1)t)}\right)^{d-1} e^{-z^2}\ \d z\\[10pt]
		&\geq 2\sqrt{t} \int_{-\frac{\sqrt{\lambda_1 t}}{2}}^{\infty} \left(1- e^{-2(2\sqrt{t}z+(d-1)t)}\right)^{d-1} e^{-z^2}\ \d z\\[10pt]
		& \geq  2\sqrt{t} \left(1- e^{-(d-1)t}\right)^{d-1} \int_{-\frac{\sqrt{\lambda_1 t}}{2}}^{\infty}  e^{-z^2}\ \d z\\[10pt]
		&= 2\sqrt{t} \left(1- e^{-(d-1)t}\right)^{d-1}\left[ \int_{-\infty}^{\infty}  e^{-z^2}\ \d z - \int_{-\infty}^{-\frac{\sqrt{\lambda_1 t}}{2}}  e^{-z^2}\ \d z \right]\\[10pt]
		&= 2\sqrt{t} \left(1- e^{-(d-1)t}\right)^{d-1}\left[\sqrt{\pi} - \int_{-\infty}^{-\frac{\sqrt{\lambda_1 t}}{2}}  e^{-z^2}\ \d z \right]\\[10pt]
		&\geq 2\sqrt{t} \left(1- e^{-(d-1)t}\right)^{d-1}\left[\sqrt{\pi} +\frac{1}{\sqrt{\lambda_1 t}} \int_{-\infty}^{-\frac{\sqrt{\lambda_1 t}}{2}} 2z e^{-z^2}\ \d z \right]\\[10pt]
		&= 2\sqrt{t} \left(1- e^{-(d-1)t}\right)^{d-1}\left[\sqrt{\pi} - \frac{1}{\sqrt{\lambda_1 t}}e^{-{\frac{\lambda_1 t}{4}}} \right].
	\end{aligned}
	\]
	In total, we arrive to
	\[
	\begin{aligned}
		\frac{\partial_t C(t)}{C(t)}&\geq \frac{(d-1)^2t+2}{4t} - \frac{(d-1)^2\sqrt{\pi}\ t^2  +2\sqrt{\pi}\ t  + (d-1)t^\frac{3}{2}e^{-\lambda_1 t}}{4t^2\left(1-e^{-(d-1)t}\right)^{d-1} \left(\sqrt{\pi} - \frac{1}{\sqrt{\lambda_1 t}}e^{-\frac{\lambda_1 t}{4}}\right)}\\[10pt]
		&=\frac{(d-1)^2t+2}{4t}\cdot h(t) - \frac{\left(\frac{2\sqrt{t}}{\sqrt{\lambda_1}}+2(d-1)t^{\frac{3}{2}}\right)e^{-\frac{\lambda_1 t}{4}} + (d-1)t^\frac{3}{2}e^{-\lambda_1 t} }{4t^2\left(1-e^{-(d-1)t}\right)^{d-1} \left(\sqrt{\pi} - \frac{1}{\sqrt{\lambda_1 t}}e^{-\frac{\lambda_1 t}{4}}\right)},
	\end{aligned}
	\]
	where
	\begin{equation}\label{eq:def_h(t)}
		h(t):= \left(1-\left(1-e^{-(d-1)t}\right)^{-(d-1)}\right)<0.
	\end{equation}
	Since
	\[
	\lim\limits_{t\to\infty} \frac{h(t)}{-(d-1)e^{-(d-1)t}} = 1,
	\]
	it is left as an exercise for the reader to check that 
	\[
	h(t)\geq -2(d-1)e^{-(d-1)t}\quad\text{for all}\quad t\geq 1.
	\]
	Therefore, for $t\geq 1$,
	\[
	\frac{\partial_t C(t)}{C(t)}\geq -\frac{(d-1)^3t+2(d-1)}{2t} e^{-(d-1)t} - \frac{\left(\frac{2\sqrt{t}}{\sqrt{\lambda_1}}+2(d-1)t^{\frac{3}{2}}\right)e^{-\frac{\lambda_1 t}{4}} + (d-1)t^\frac{3}{2}e^{-\lambda_1 t} }{4t^2\left(1-e^{-(d-1)t}\right)^{d-1} \left(\sqrt{\pi} - \frac{1}{\sqrt{\lambda_1 t}}e^{-\frac{\lambda_1 t}{4}}\right)}.
	\]
	Thus, if we define 
	\begin{equation}\label{eq:def_m_d}
		m_d := \min \left(d-1, \frac{(d-1)^2}{16}\right),
	\end{equation}
	then there must exist a positive constant $k_d$ such that
	\[
	\frac{\partial_t C(t)}{C(t)}\geq -k_d\ e^{-m_d t} \quad\text{for all}\quad t\geq 1.
	\]
	
	Now, for the upper bound  for ~\eqref{eq:quotient_C(t)}. For the sake of brevity, we will omit computations similar to the ones in the previous step. We estimate again each integral separately. First, by similar arguments as before,
	\[
	\mathcal{I}\geq 2\sqrt{t}\left(1-e^{-(d-1)t}\right)^{d-1}\int_{-\frac{\sqrt{\lambda_1 t}}{2}}^\infty \left(4tz^2 +(d-1)^2t^2  +4(d-1)t^\frac{3}{2}  \right) e^{-z^2}\ \d z.
	\]
	We use now the following bound, left for the reader as an exercise. For every $a>0$,
	\[
	\int_{-a}^\infty e^{-z^2}\ \d z = \int_{-\infty}^\infty e^{-z^2}\ \d z - \int_{-\infty}^{-a} e^{-z^2}\ \d z\geq \sqrt{\pi} - \frac{1}{2a}e^{-a^2}.
	\]
	Using this, and that $\sqrt{\lambda_1}=(d-1)/2$ by definition, it is not hard to arrive to
	\[
	\mathcal{I}\geq 2\sqrt{t}\left(1-e^{-(d-1)t}  \right)^{d-1}\left[\sqrt{\pi}(d-1)^2t^2+2\sqrt{\pi}t + e^{-\frac{\lambda_1 t}{4}}\left(   \frac{d-1}{2}t^\frac{3}{2} - \frac{2}{\sqrt{\lambda_1}}\sqrt{t}  \right)  \right].
	\]
	For the integral $\mathcal{II}$ we simply compute
	\[
	\mathcal{II}\leq 2\sqrt{t}\int_{-\sqrt{\lambda_1 t}}^\infty e^{-z^2}\ \d z\leq 2\sqrt{t}\int_{-\infty}^\infty e^{-z^2}\ \d z = 2\sqrt{\pi t}.
	\]
	Substituting in~\eqref{eq:quotient_C(t)} we arrive, after some computations, to
	\[
	\frac{\partial_t C(t)}{C(t)}\leq \left( \frac{(d-1)^2}{4} + \frac{1}{2t} \right)h(t) + \left(1-e^{-(d-1)t}  \right)^{d-1} \left( \frac{d-1}{2}t^\frac{3}{2} - \frac{2}{\sqrt{\lambda_1}}\sqrt{t} \right) \frac{e^{-\frac{\lambda_1 t}{4}}}{4t^2\sqrt{\pi}},
	\]
	where $h(t)$ is defined as in~\eqref{eq:def_h(t)}. From here it follows that
	\[
	\frac{\partial_t C(t)}{C(t)}\leq \left( \frac{(d-1)^3}{2} + \frac{1}{t} \right)e^{-(d-1)t} + \left( \frac{d-1}{2}t^\frac{3}{2} - \frac{2}{\sqrt{\lambda_1}}\sqrt{t} \right) \frac{e^{-\frac{\lambda_1 t}{4}}}{4t^2\sqrt{\pi}},
	\]
	and from here, with $m_d$ defined as in~\eqref{eq:def_m_d},
	\[
	\frac{\partial_t C(t)}{C(t)}\leq k_d e^{-m_d t}\quad\text{for all}\quad t>1.
	\]
	Putting all together we obtain~\eqref{eq:bound_quotient_C(t)}.
	
	Along the previous steps, we have proven that
	\[
	2\sqrt{t} \left(1- e^{-(d-1)t}\right)^{d-1}\left[\sqrt{\pi} - \frac{1}{\sqrt{\lambda_1 t}}e^{-{\frac{\lambda_1 t}{4}}} \right]\leq  \mathcal{II}\leq  2\sqrt{\pi t}
	\]
	With this bounds and equation~\eqref{eq:equality_C_radial_solutions} it is easy to conclude~\eqref{eq:limit_C(t)}.
\end{proof}

\subsection{Explicit expression for the function $\varphi$}\label{sect:ExplicitVarPhi}
%\ag{We have to discuss this result. We proved it for $d=2, 3$, and verified it on paper for $d=5$. For $d>3$ our claim seems true, but for now, I don't know how to prove it, the actual proof is wrong. Maybe we can spend some time translating the result of Anker-Papageorgiou-Zhang (which is written in a super-general way with a lot of geometric terms) to ours, but we can also add a remark with our candidate for the formula saying that APZ proved the existence of such limit and computed it for the interested reader, but we didn't want to check if our formula and theirs is the same, since that is not the focus of the article. In other words, here is the limit in $d=2, 3$, and we believe we have also this formula for $d>3$, but if you don't believe us and need the formula, go to APZ.}
This section is devoted to the proof of the explicit expression for the function $\varphi$ defined in \eqref{eq:limitQuotient} in the cases $d\in\{2,3\}$. Recall that we denote $(r_y,\theta_y)$ for the polar coordinates centred in some origin $\mathcal{O} \in \mathbb{H}^d$ of some point $y \in \mathbb{H}^d$, and we write $\lambda (x,y) = \textup{dist}(x,y)$. The following result is a direct, although long computation.
\begin{prp}\label{propVarPhi}
	Fix $y\in\H^d$ and $l\in(0,1)$. Given any $\ell \in (-l  \sqrt{t}, l \sqrt{t})$, define $r_{\ell} (t) = (d-1)t+ \ell$ for all $t \in (0,+\infty)$. Then, for any $\theta \in \mathbb{S}^{d-1}$ we have that $p_t=p_t(r_\ell(t), \theta)$ denotes a unique point in $\H^d$ and	there exists a continuous function $\varphi (y,\theta)$ such that
	\begin{equation*}
		\displaystyle\lim_{t \to +\infty} \frac{P_d^y (t,p_t)}{ P_d (t,p_t)} = \varphi (y,\theta),
	\end{equation*}
	where, in the particular cases $d=2, 3$,
	\begin{equation*}
		\begin{aligned}
			\varphi (y,\theta) =\left[ \cosh (r_y)-\sinh (r_y) \cdot \langle\theta,\theta_y\rangle \right]^{-(d-1)}.
		\end{aligned}
	\end{equation*}
\end{prp}
\begin{proof}

	Write $\lambda (p_t,y) = \textup{dist} (p_t, y)$. Then, by the second hyperbolic law of cosines we have that
	\begin{equation}\label{ry1}
		\lambda (p_t,y)  = \textup{arccosh} \left( \cosh (r_\ell(t)) \cosh (r_y)-\sinh (r_\ell(t)) \sinh (r_y) \cdot \langle\theta,\theta_y\rangle \right).
	\end{equation}
	Since the fundamental solution depends only on the distance to the point at its centre, $P^y_d(t,x)=G_d(t,\lambda)$, the proof of the result reduces to compute the limit
	\begin{equation}\label{hd}
		\displaystyle\lim_{t \to +\infty} \frac{G_d (t,\lambda (p_t,y))}{G_d (t,r_\ell (t))}.
	\end{equation}

	Let us denote the argument of the arccosh in \eqref{ry1} by
	\[
	\alpha:= \cosh (r_\ell(t)) \cosh (r_y)-\sinh (r_\ell(t)) \sinh (r_y) \cdot \langle\theta,\theta_y\rangle
	\]
	%    Since the point $y$ and the direction $\theta$ are fixed, $\alpha$ is only a function of $r_\ell(t)$.
	By using the Taylor expansion of $\textup{arccosh} (x)$, we conclude that 
	\[
	\lambda (p_t,y) = \ln (2\alpha )-\displaystyle\sum_{n=1}^{+\infty} \frac{(2n)!}{2^{2n} (n!)^2} \frac{1}{2n \alpha^{2n}}.
	\]
	Note that
	\[
	\alpha \geq \cosh (r_\ell(t)) \cosh (r_y)-\sinh (r_\ell(t)) \sinh (r_y) = \cosh(r_\ell(t)-r_y)\geq 1,
	\]
	so the above expansion is well defined. We can estimate both terms of this sum. For the first one, we have that
	\[
	\begin{split}
		\ln (2\alpha) &= \ln \left( (\cosh (r_y)-\sinh (r_y) \langle\theta,\theta_y\rangle) e^{r_\ell(t)} \right) + \ln \left( 1+ \frac{\cosh (r_y)+\sinh (r_y) \langle\theta,\theta_y\rangle}{\cosh (r_y)-\sinh (r_y) \langle\theta,\theta_y\rangle} e^{-2r_\ell(t)} \right) \\
		&= \ln \left( (\cosh (r_y)-\sinh (r_y) \langle\theta,\theta_y\rangle) e^{r_\ell(t)} \right) + O(e^{-2r_\ell(t)}).
	\end{split}
	\]
	For the second one, choose $t$ big enough such that $r_y<r_\ell(t)/2$, implying that
	\[
	\alpha\geq \cosh(r_\ell(t)-r_y)> \cosh(r_\ell(t)/2)\sim e^{\frac{r_\ell(t)}{2}}
	\]
	and therefore
	\[
	\displaystyle\sum_{n=1}^{+\infty}  \frac{(2n)!}{2^{2n} (n!)^2} \frac{1}{2n \alpha^{2n}} < \displaystyle\sum_{n=1}^{+\infty} \frac{1}{2n \alpha^{2n}} < \displaystyle\sum_{n=1}^{+\infty} \left(\frac{1}{\alpha^2}\right)^n = \frac{1}{\alpha^2-1} = O(e^{-r_\ell(t)}).
	\]
	Thus we conclude that for any $t$ big enough such that $r_y<r_\ell(t)/2$ we have 
	\begin{equation}\label{ry2}
		\begin{aligned}
			\lambda(p_t,y) &= \ln \left( (\cosh (r_y)-\sinh (r_y) \langle\theta,\theta_y\rangle) e^{r_\ell(t)} \right) + O(e^{-r_\ell(t)})\\[10pt]
			& =r_\ell(t) + \ln \left( (\cosh (r_y)-\sinh (r_y) \langle\theta,\theta_y\rangle)  \right) + O(e^{-r_\ell(t)}).
		\end{aligned}
	\end{equation}
	With this estimate for $\lambda$, we can compute the limit appearing in \eqref{hd}. We will do it only for the cases $d=2,3$. The other cases follows using \eqref{eq:recurrence_derivative_kernels}. 
	
	\noindent\textbf{Case $d=3$}.
	
	In this case, we have
	\[
	G_3 (t,r) = \frac{1}{(4 \pi t)^{\frac{3}{2}}} \frac{r}{\sinh (r)} e^{-t-\frac{r^2}{4 t}},
	\]
	and by \eqref{ry2} it is straightforward to show that
	\[
	\displaystyle\lim_{t \to +\infty} \frac{\lambda(p_t, y)}{r_\ell (t)} = 1,\quad \displaystyle\lim_{t \to +\infty} \frac{\sinh (r_\ell (t))}{\sinh (\lambda(p_t, y))} = [\cosh (r_y)-\sinh (r_y) \langle\theta,\theta_y\rangle]^{-1}
	\]
	and, using that in dimension $d=3$ we have $r_\ell(t)= 2t+\ell$, that
	\[
	\displaystyle\lim_{t \to +\infty} \frac{e^{\frac{r_\ell (t)^2}{4t}}}{e^{\frac{\lambda(p_t,y)^2}{4 t}}} = [\cosh (r_y)-\sinh (r_y) \langle\theta,\theta_y\rangle]^{-1},
	\]
	so it follows that
	\[
	\displaystyle\lim_{t \to +\infty} \frac{G_3 (t,\lambda(p_t, y))}{G_3 (t,r_\ell (t))} = (\cosh (r_y)-\sinh (r_y) \langle\theta,\theta_y\rangle)^{-2}.
	\]
	
	\noindent\textbf{Case $d=2$.}
	
	Now we have 
	\[
	G_2 (t,r) = \frac{\sqrt{2}}{(4 \pi t)^{\frac{3}{2}}}e^{-\frac{1}{4}t} \int_{r}^{\infty} \frac{s e^{-\frac{s^2}{4 t}}}{\sqrt{\cosh (s) -\cosh (r)}} \, ds,
	\]
	so we cannot compute the limit directly. Instead, we check first the asymptotics of $G_2 (t, r_\ell (t))$. It is enough to compute the asymptotics of the integral
	\[
	I(t) = \int_{r_\ell}^{\infty} \frac{s e^{-\frac{s^2}{4 t}}}{\sqrt{\cosh (s) -\cosh (r_\ell)}} \, ds =  \int_{0}^{\infty} \frac{(z+r_\ell) e^{-\frac{(z+r_\ell)^2}{4 t}}}{\sqrt{\cosh (z+r_\ell) -\cosh (r_\ell)}} \, dz,
	\] 
	where we did the change of variables $z=s-r$. We divide this integral in two. First we compute
	\[
	J(t) = \int_{0}^{\sqrt{t}} \frac{(z+r_\ell) e^{-\frac{(z+r_\ell)^2}{4 t}}}{\sqrt{\cosh (z+r_\ell) -\cosh (r_\ell)}} \, dz = \frac{1}{\sqrt{\sinh (r_\ell)}}\int_{0}^{\sqrt{t}} \frac{(z+r_\ell) e^{-\frac{(z+r_\ell)^2}{4 t}}}{\sqrt{z \cdot f(t,z)}} \, dz, 
	\]
	with $$f(t,z) = \frac{1}{z} \left(\frac{\cosh (z+r_\ell (t))}{ \sinh (r_\ell (t))}- \coth (r_\ell (t)) \right)= \frac{1}{z} (e^z-1 + o(t)),$$
	where we write $o(t)$ for a function such that $o(t) \to 0$ as $t \to \infty$. Now, we note that
	\[
	\begin{split}
		\int_{0}^{\sqrt{t}} r_\ell \, z^{-1/2} e^{-\frac{(z+r_\ell)^2}{4 t}} \, dz &= r_\ell e^{-\frac{r_\ell^2}{4 t}} \left(\frac{2t}{r_\ell}\right)^{1/2} \int_0^V e^{-v} e^{-\frac{tv^2}{r_\ell^2}} v^{-1/2} \, dv \\
		& = e^{-1/2} \sqrt{2 \pi r_\ell t} e^{-\frac{r_\ell^2}{4t}} (1+ o (t)),
	\end{split}
	\]
	To check this identity, note that $r_\ell (t) =t+l$ and we have done the change of variables $v=(z r_\ell) / (2t)$. On the other hand, we have that
	\[
	\displaystyle\lim_{t \to \infty} \int_0^{\sqrt{t}}\frac{ \sqrt{z}}{f(t,z)} e^{-\frac{z^2}{4 t}-\frac{r_\ell z}{2 t}} \, dz = \int_{0}^{+\infty} \frac{e^{-z} z}{\sqrt{e^z-1}} \, dz <+\infty,
	\]
	so we conclude that
	\[
	J(t) = C_1 \sqrt{2\pi r_\ell t} e^{-\frac{r_\ell^2}{4 t}} (1+o(t)),
	\]
	for some $C_1>0$. 
	
	Now we deal with the integral
	\[
	E(t) = \int_{\sqrt{t}}^{+\infty} \frac{(z+r_\ell) e^{-\frac{(z+r_\ell)^2}{4 t}}}{\sqrt{\cosh (z+r_\ell) -\cosh (r_\ell)}} \, dz.
	\]
	In this case, we note that
	\[
	\cosh (z+r_\ell)-\cosh (r_\ell) \geq \cosh (\sqrt{t}+r_\ell)-\cosh (r_\ell) = e^{r_\ell+\sqrt{t}}+ o(t),
	\]
	so we get that
	\[
	E(t) \leq \frac{1}{e^{r_\ell+\sqrt{t}}+o(t)} \int_{\sqrt{t}}^{\infty} (z+r_\ell) e^{-\frac{(z+r_\ell)^2}{4 t}}\, dz = \frac{C_2 t e^{-\frac{r_\ell^2}{4t}-\frac{r_\ell}{2 t}}}{e^{r_\ell+\sqrt{t}}+o(t)},
	\]
	and this term in negligible with respect to $J(t)$. Thus, we conclude that
	\[
	G_2 (t, r_\ell (t)) = C_3\left(\frac{\sqrt{r_\ell (t)} e^{-\frac{r_\ell(t)^2}{4t}-\frac{t}{4}}}{ \sqrt{t}\sqrt{\sinh (r_\ell (t))}}+ o(t)\right),
	\]
	with $C_3>0$. 
	
	Now it is straightforward that
	\[
	\displaystyle\lim_{t \to +\infty} \frac{G_2 (t,\lambda(p_t, y))}{G_2 (t,r_\ell (t))} = (\cosh (r_y)-\sinh (r_y) \langle\theta,\theta_y\rangle)^{-1}.
	\]

\end{proof}

\subsection*{Acknowledgments, COI, Data availability}

J.~Cañizo and A.~Gárriz were supported by grants PID2023-151625NB-I00, RED2024-153842-T, and the IMAG María de Maeztu
grant CEX2020-001105-M, all of them funded by MICIU/AEI/10.13039/501100011033.

A.~Gárriz was also supported by the Juan de la Cierva grant
JDC2022-048653-I, funded by MICIU/AEI/10.13039/501100011033 and by the Ramón y Cajal grant RYC2024-050952-I, funded by MICIU/AEI/10.13039/501100011033 and by the FSE+.

D.A. Marín was supported by the {\it Maria de Maeztu} Excellence Unit IMAG, reference CEX2020-001105-M, funded by MCINN/AEI/10.13039/ 501100011033/CEX2020-001105-M, and Spanish MIC Grant PID2023-150727NB.I00.

\medskip

\textbf{Conflict of Interests: } The authors declare that they have no conflict of interests of any kind.

\medskip

\textbf{Data availability statement: } The authors declare that there is no external associated data to be considered for this article.

\bibliography{document}

\end{document}